\documentclass[a4paper,10pt]{article}
\usepackage{amsmath,amssymb,amsfonts,amsthm,mathrsfs}
\usepackage{hyperref}
\usepackage{xcolor}
\usepackage{tikz-cd}
\usepackage[all,cmtip]{xy}
\usepackage[style=alphabetic,sorting=nyt,maxbibnames=7,maxcitenames=5, backend=biber, style=alphabetic]{biblatex}
\newtheorem{thm}{Theorem}[section]
\newtheorem{mainthm}{Theorem}

\newtheorem{mainprop}[mainthm]{Proposition}

\newtheorem{cor}[thm]{Corollary}
\newtheorem{lem}[thm]{Lemma}
\newtheorem{prop}[thm]{Proposition}

\theoremstyle{definition}
\newtheorem{dfn}[thm]{Definition}
\newtheorem{rem}[thm]{Remark}
\AtEndEnvironment{rem}{\null\hfill\ensuremath{\triangle}}
\newtheorem{ques}[thm]{Question}

\newcommand{\into}{\hookrightarrow}
\newcommand{\onto}{\twoheadrightarrow}
\newcommand{\slcpt}{/ \!\! /}

\newcommand{\TopS}{\Sigma_{\mathrm{top}}}

\newcommand{\RR}{\mathbb{R}}
\newcommand{\NN}{\mathbb{N}}
\newcommand{\ZZ}{\mathbb{Z}}
\newcommand{\QQ}{\mathbb{Q}}
\newcommand{\Sph}{\mathbb{S}}

\newcommand{\C}{\mathcal{C}} % poset of compact subspaces
\newcommand{\E}{\operatorname{E}} % free simplicial set over G
\newcommand{\SD}{\operatorname{SD}} % double subdivision

\newcommand{\Sym}{\operatorname{Sym}}
\newcommand{\BS}{\operatorname{BS}} % Baumslag--Solitar
\newcommand{\Hom}{\operatorname{Hom}}
\newcommand{\Homtop}{\operatorname{Hom}_{\mathrm{top}}}
\renewcommand{\P}{\mathcal P} % a finite set of primes
\newcommand{\B}{\mathcal B} % group of upper triangular matrices in SL_n.
\newcommand{\Ch}{\operatorname{C}} % simplicial chain complex
\newcommand{\h}{\operatorname{H}} % homology
\newcommand{\redH}{\tilde{\operatorname{H}}} % reduced homology
\newcommand{\VR}{\operatorname{VR}} % Vietoris-Ripps complex
\newcommand{\ev}{\operatorname{ev}_\Lambda} % the map from the Schlichting completion to the coset space
\newcommand{\G}{\mathcal G} % In the main proof this will abbreviate the Schlichting completion.
\renewcommand{\L}{\mathcal L}

\title{Geometric invariants of TDLC completions}
\author{Ilaria Castellano$^1$ \and José Pedro Quintanilha$^2$}

\date{$^1$Heinrich-Heine-Universität Düsseldorf\\
	$^2$Ruprecht-Karls-Universität Heidelberg\\\bigskip
	\today}

\begin{document}

\maketitle
\begin{abstract}
   Recently, Bonn and Sauer showed that, from the point of view of compactness properties, the Schlichting completion of a Hecke pair $(\Gamma,\Lambda)$ behaves precisely as if it were the quotient of $\Gamma$ by $\Lambda$. Motivated by this result, we prove that a similar phenomenon holds for the $\Sigma$-sets. More generally, we relate the $\Sigma$-sets of every TDLC~completion  of a Hecke pair $(\Gamma,\Lambda)$ to the $\Sigma$-sets of $\Gamma$ whenever $\Lambda$ satisfies suitable compactness properties. We provide applications to TDLC~completions of Baumslag--Solitar groups and certain groups of upper triangular matrices studied by Schesler.
\end{abstract}

\tableofcontents

\section{Introduction}
Recently, a theory of $\Sigma$-sets for locally compact Hausdorff
groups has been introduced and studied by Bux, Hartmann, and the second author \cite{BHQ24a,BHQ24b}. This novel theory, on the one hand, refines the compactness properties $\mathrm C_n$ and $\mathrm {CP}_n(R)$ of Abels and Tiemeyer (where $R$~is a commutative ring) \cite{AT97}, and
on the other hand, recovers the classical theory of $\Sigma$-sets\footnote{In the
literature, $\Sigma$-sets are also sometimes called ``geometric invariants'' or “BNSR-invariants”, due to Bieri,
Neumann, Strebel and Renz, who pioneered the theory \cite{Ren88,BR88}.} for discrete groups.

The principal area of application for the theory of compactness properties and, by extension, these newly defined $\Sigma$-sets are totally disconnected locally compact (TDLC) groups \cite{CCC20,CW25}. %This class includes discrete groups, profinite groups, automorphism groups of locally finite graphs (or other locally finite geometric structures) and matrix groups over non-Archimedean local fields. 
A notable source of examples are TDLC completions of Hecke pairs. A {Hecke pair} $(\Gamma, \Lambda)$ consists of a Hausdorff topological group~$\Gamma$ and an open commensurated subgroup $\Lambda\subseteq\Gamma$. The property of $\Lambda$~being commensurated is less restrictive than being normal, so one cannot define a quotient group~$\Gamma / \Lambda$. Still, there are TDLC groups~$G$ with continuous homomorphisms $\Gamma \to G$ with dense image for which $\Lambda$~is the preimage of a compact open subgroup $L$ of~$G$. 

As explained by Reid and Wesolek \cite{RW19}, the Schlichting completion $\Gamma \slcpt \Lambda$ (together with its completion map $\alpha \colon \Gamma \to \Gamma \slcpt \Lambda$) is, in a precise sense, the smallest such TDLC group (see Theorem~\ref{thm:compl}), and if $\Lambda$~happens to be normal, then one simply has $\Gamma \slcpt \Lambda = \Gamma / \Lambda$.
Recently, Bonn and Sauer studied the behaviour of the compactness properties under taking Schlichting completions, and showed that, from this perspective, $\Gamma \slcpt \Lambda$ behaves exactly as though it were a quotient of~$\Gamma$ by~$\Lambda$ \cite[Theorems 1.1~and~1.2]{BS24}.
This motivates investigating whether the same phenomenon holds for $\Sigma$-sets -- more precisely, whether a relationship between the $\Sigma$-sets of~$\Gamma$ and of~$\Gamma \slcpt \Lambda$ may be inferred from compactness properties of~$\Lambda$.

The $\Sigma$-sets of a locally compact Hausdorff group~$G$ are certain subsets of the space $\Homtop(G,\RR)$ of continuous homomorphisms $G\to \RR$ (usually called ``characters''). Thus, in comparing $\Sigma$-sets of $\Gamma$ and~$\Gamma \slcpt \Lambda$, one might first like to relate their character spaces. We begin by showing that there is in fact a simple connection between the character spaces of~$\Gamma$ and of any TDLC completion of $(\Gamma, \Lambda)$:

\begin{mainprop}[Character spaces of TDLC completions]\label{prop:char_space_intro}
Let $\phi\colon \Gamma \to G$ be a TDLC completion of a Hecke Pair $(\Gamma, \Lambda)$. Then precomposition with $\phi$ induces an isomorphism of $\RR$-vector spaces
   $$\phi^*\colon \Homtop(G, \RR) \xrightarrow{\cong} \Homtop(\Gamma, \RR)_\Lambda,$$
where $\Homtop(\Gamma, \RR)_\Lambda$ denotes the space of characters $\Gamma \to \RR$ that vanish on~$\Lambda$.
\end{mainprop}

This statement is presented in the text as Corollary~\ref{cor:char_space}. It contextualizes our main result, where we show that indeed Bonn and Sauer's result extends to the $\Sigma$-theory of all TDLC completions of the Hecke pair~$(\Gamma,\Lambda)$:

\begin{mainthm}[$\Sigma$-sets of TDLC completions]\label{thm:main_intro}
Fix $n\in \NN$ and a commutative ring~$R$.
  Let $\phi\colon \Gamma \to G$~ be a TDLC completion of a Hecke pair~$(\Gamma, \Lambda)$, with $\Gamma$ a locally compact Hausdorff group,
  let $\bar\chi \colon G \to \RR$ be a character and $\chi:=\bar\chi \circ \phi$.

    If $\Lambda$ is of type~$\mathrm{C}_n$, then:
    \begin{enumerate}
        \item for every $k\le n$, if $\bar\chi \in \TopS^k(G)$, then $\chi \in \TopS^k(\Gamma)$,
        \item for every $k\le n+1$, if $\chi \in \TopS^k(\Gamma)$, then $\bar\chi \in \TopS^k(G)$.
    \end{enumerate}

    Similarly, if $\Lambda$~is of type~$\mathrm{CP}_n(R)$, then:
     \begin{enumerate}
        \item for every $k\le n$, if $\bar\chi \in \TopS^k(G;R)$, then $\chi \in \TopS^k(\Gamma;R)$,
        \item for every $k\le n+1$, if $\chi \in \TopS^k(\Gamma;R)$, then $\bar\chi \in \TopS^k(G;R)$.
    \end{enumerate}
\end{mainthm}

Theorem~\ref{thm:main_intro} appears in the text as Theorem~\ref{thm:main}. In light of Proposition~\ref{prop:char_space_intro}, we see that it covers all characters of~$G$. In particular, by specializing to the character $\chi=0$ and the completion $G=\Gamma\slcpt\Lambda$, it recovers Bonn--Sauer's theorem, providing an alternative proof of their result.

Broadly speaking, the technical core of the article is divided into two parts. First, we connect the membership status of the character~$\chi\colon \Gamma\to \RR$ to the essential connectedness / acyclicity properties of a certain filtration associated to the space of left cosets~$\Gamma /\Lambda$. The argument here is essentially adapted from analogous results relating the $\Sigma$-theory of locally compact Hausdorff groups in a short exact sequence \cite[Theorem~8.2]{BHQ24a} \cite[Theorem~J]{BHQ24b}. However, one new idea is needed in order to deal with the steps that relied on the kernel group being normal (whose role is now played by~$\Lambda$).
The key observation is that $\Lambda$~enjoys a weaker property, which we have termed ``almost commutativity with compacts'' (Definition~\ref{dfn:almost_comm}), which still enables the argument to go through.
The resulting Propositions \ref{prop:sigmaofquotient}~and~\ref{prop:sigmaofquotient_homological}, generalize those earlier theorems by relaxing the assumption that the kernel group is normal, to closed subgroups that almost commute with compacts.

The second part of the proof relates the aforementioned filtration associated to~$\Gamma / \Lambda$, to the filtration defining membership of $\bar\chi$ in the $\Sigma$-sets of~$G$, by means of two simplicial homotopy equivalences. We first do this in the case $G=\Gamma \slcpt \Lambda$, and then a quick argument extends the result to all TDLC completions of $(\Gamma, \Lambda)$.

After concluding the proof, we provide two applications of Proposition~\ref{prop:char_space_intro} and Theorem~\ref{thm:main_intro}. First, we give a complete description of the $\Sigma$-sets of the Schlichting completions~$G_{m,n}$ of the Baumslag--Solitar groups $\BS(m,n)$ (these completions were earlier studied by Elder and Willis \cite{EW18}). Second, we give an explicit description of the Schlichting completions of certain matrix groups studied by Schesler \cite{S23}, and extend his partial results on their $\Sigma$-theory to these completions.

\subsection*{Structure of the article}

Section~\ref{s:preliminaries} recalls the basic notions and tools used throughout the paper, including some results on simplicial sets (which lie at the core of the $\Sigma$-theory of locally compact groups), the definitions and basic properties of $\Sigma$-sets for locally compact Hausdorff groups, and an overview of TDLC completions of Hecke pairs, with emphasis on the Schlichting completion.

The key property of ``almost commutativity with compacts'' is introduced in Section~\ref{sec:almost_comm}, and we show that it is satisfied by open commensurated subgroups (Proposition~\ref{prop:comm_exchange}). We also ask whether it holds, more generally, for closed commensurated subgroups and provide an affirmative answer, due to George Willis, in the case where the ambient group is TDLC (Proposition~\ref{prop:Willis}).

Character spaces of TDLC completions are studied in Section~\ref{sec:char_space}, where we prove Proposition~\ref{prop:char_space_intro} (as Corollary~\ref{cor:char_space}) and also show that $\Sigma$-sets are essentially the same across all TDLC completions of a Hecke pair (Proposition~\ref{prop:completions_have_same_Sigma}). 

The proof of our main result, Theorem~\ref{thm:main_intro}, is the subject of Section~\ref{sec:main_proof}, which is divided into two parts as outlined earlier. Section~\ref{sec:part1} connects the $\Sigma$-sets of $\Gamma$ and the homotopy type of a filtration associated to the coset space $\Gamma /\Lambda$, for $\Lambda$~a closed subgroup that almost commutes with compacts. This is done separately in the homotopical and homological setup, culminating in Propositions~\ref{prop:sigmaofquotient}
and~\ref{prop:sigmaofquotient_homological}, respectively. Then, in Section~\ref{sec:part2}, we provide homotopy equivalences from this filtration to the one defining membership of the relevant character in the $\Sigma$-sets of the Schlichting completion. Once this is done, we prove our main result, as Theorem~\ref{thm:main}.

In Section~\ref{sec:applications} we give two applications of our results. In Section~\ref{sec:BS} we compute all $\Sigma$-sets of the Schlichting completions of Baumslag--Solitar groups (Proposition~\ref{prop:BS_main}), and in Section~\ref{sec:Schesler} we extend Schesler's partial computations of $\Sigma$-sets of certain matrix groups to their Schlichting completions (Proposition~\ref{prop:Schesler_main}).

In this article, we opted to use a definition of the homological $\Sigma$-sets that closely parallels the one of homotopical $\Sigma$-sets, making the similarities between the two theories more transparent, and many arguments more uniform. However, this choice deviates somewhat from the original definition \cite{BHQ24b}, so we explain the translation between the two descriptions in Appendix~\ref{sec:appendix}.

\subsection*{Acknowledgements}
We thank Colin Reid, Roman Sauer and George Willis for kindly answering our questions during the preparation of this work. George Willis also supplied the proof of Proposition~\ref{prop:Willis}.

\section{Preliminaries}\label{s:preliminaries}

\subsection{Simplicial sets}\label{sec:SSets}

Simplicial sets play a central role in the definition of $\Sigma$-sets for locally compact groups. The basics of the theory of simplicial sets will be assumed, which the reader may consult in Friedman's introductory survey \cite{Fri12}. For a more extensive treatment see for example Curtis's article \cite{Cur71}. Our ``simplicial toolkit'' will also need part of the material collected in an earlier article \cite[Section~2]{BHQ24a}, whose main points we now summarize.

\subsubsection{Free simplicial sets}

The \textbf{free simplicial set}~$\E X$ on a set~$X$ has as $k$-simplices all the $(k+1)$-tuples $(x_0,\ldots, x_k) \in X^{k+1}$. The $i$-th face and degeneracy maps are given by, respectively, suppressing or doubling the $i$-th entry:
    \begin{align*}
    d_i(x_0, \ldots, x_k)&=(x_0, \ldots, x_{i-1}, x_{i+1}, \ldots, x_k)\\
    s_i(x_0, \ldots, x_k)&=(x_0, \ldots, x_{i-1}, x_{i}, x_{i}, x_{i+1},\ldots, x_k).
    \end{align*}

In particular, the $0$-skeleton is given by $\E X^{(0)} = X$, and it is clear that every action of a group~$G$ on~$X$ extends to an action on~$\E X$. Of particular importance here is the case where $X=G$. Then, the free simplicial set~$\E G$ carries a $G$-action, whose induced action on the geometric realization~$|\E G|$ is free. Note however that in case $G$~has a topology, this is not an action of~$G$ as a topological group, but only of the underlying abstract group.

\subsubsection{Homotopy groups}

We will regard simplicial sets as topological objects by means of the geo\-metric realization functor~$Y \mapsto |Y|$, and we will loosely talk about topological properties of a pointed simplicial set~$(Y,y)$ when in rigor we mean $(|Y|,y)$. Accordingly, we also refer to the \textbf{homotopy groups}~$\pi_k(Y, y):= \pi_k(|Y|, y)$, which are typically defined only if $Y$~is a Kan complex, even when this is not the case. %This definition $\pi_k(Y,y) := \pi_k(|Y|,y)$ is compatible with the usual one for Kan complexes ultimately because the canonical map $Y \to \operatorname{Sing}(|Y|)$ to the singular set of its topological realization is a weak equivalence \cite[Proposition~I.11.1]{GJ99}.

Despite this topological notion of homotopy groups, one can return to the combinatorial world of simplicial sets by means of the semisimplicial approximation theorem, a tool analogous to simplicial approximation in the theory of simplicial complexes.
One of the main ingredients is a certain ``double barycentric subdivision functor''~$\SD\colon \mathbf{SSet}\to \mathbf{SSet}$, which has the property that for every simplicial set~$Y$, there exists a natural map $\Phi \colon \SD(Y) \to Y$ whose realization~$|\Phi|$~is homotopic to a homeomorphism $|\SD(Y)|\cong|Y|$.
Fixing once and for all, for each $k\in \NN$, a homeomorphism $\Sph^k \cong |\partial\Delta^{k+1}|$ that identifies the preferred basepoint of~$\Sph^k$ with a base-vertex of $\partial \Delta^{k+1}$, we then have \cite[Lemmas 2.4 and~2.5]{BHQ24a}:

\begin{lem}[Combinatorial $\pi_k$]\label{lem:comb_pi_k} Let $(Y,y)$ be a pointed simplicial set and let $k\in \NN$.
\begin{enumerate}
    \item For every pointed map $\eta_\circ \colon |\partial \Delta^{k+1}| \to |Y|$, there are $m\in \NN$ and a map of simplicial sets $\eta \colon \SD^m(\partial \Delta^{k+1}) \to Y$ such that $|\eta|$~is pointed-homotopic to $\eta_\circ\circ |\Phi^m|$.
    \item Given $m\in \NN$ and a pointed simplicial map $\eta \colon \SD^m(\partial\Delta^{k+1})\to Y$, the realization $|\eta|$ represents the trivial element of $\pi_k(Y,y)$ if and only if there are $m'\ge m$ and $\mu\colon \SD^{m'}(\Delta^{k+1})\to Y$ extending $\eta \circ \Phi^{m'-m}$.
\end{enumerate}
\end{lem}

In short, Lemma~\ref{lem:comb_pi_k} says that one can represent every element of $\pi_k(Y,y)$ by a simplicial map out of some subdivision of the $k$-sphere, and also certify its triviality using maps out of a subdivision of the $(k+1)$-disk.

\subsubsection{Homology $R$-modules}

Given a simplicial set~$Y$ and a commutative ring~$R$ (with $1$), the reduced homology $R$-modules $\redH_*(Y;R)$ are computed from its \textbf{augmented chain complex}~$\Ch(Y;R)$. For each $k\in \NN$, the $R$-module~$\Ch_k(Y;R)$ of $k$-chains is the free $R$-module with basis the $k$-simplices of~$Y$. There is an additional nonzero entry in degree~$-1$, with $\Ch_{-1}(Y;R):= R$. For $k\ge 1$, the boundary maps $\partial_k \colon \Ch_k(Y;R) \to \Ch_{k-1}(Y;R)$ are given on basis elements by the usual alternating sum
$$\partial_k(\sigma) = \sum_{i=0}^k (-1)^i d_i(\sigma)$$
(we will usually drop the subscript ``$k$'' from the notation~$\partial_k$). The last map, usually denoted by $\varepsilon \colon \Ch_0(Y;R) \to R$, is called the \textbf{augmentation map} and sends each vertex to~$1$. 
 We will write~$\Ch(Y;R)^{(n)}$ to denote the $n$-skeleton of this chain complex, that is, the chain complex obtained by replacing with~$0$ all $R$-modules $\Ch_k(Y;R)$ with $k\ge n+1$. 
A map of augmented $R$-chain complexes is said to ``extend $\operatorname{id} \colon R\to R$'' if the map in dimension $-1$ is the identity.

It is a standard fact that the reduced (singular) homology $R$-modules $\redH_k(|Y|;R)$ of the geometric realization are naturally isomorphic to $\redH_k(Y;R)$, which again justifies us being loose in talking about the homology modules of~$Y$ versus~$|Y|$.

\subsubsection{Simplicial homotopies}

Two maps $f,g\colon X \to Y$ of simplicial sets might have homotopic geometric realizations $|f|,|g|$ despite not being homotopic as simplicial maps. We will however encounter situations where such convenient ``simplicial'' homotopies are available, so we now briefly discuss them.

The \textbf{standard $1$-simplex}~$\Delta^1$ is the simplicial set whose $k$-simplices are the $(k+1)$-tuples of the form $(0,\dots, 0,1, \dots, 1)$, and whose $i$-th face (resp. degeneracy) map is obtained by deleting (resp. doubling) the $i$-th entry. For each $i\in \{0,\dots,k+1\}$, we denote by~$\tau_i^k$ the $k$-simplex of~$\Delta^1$ with $i$~zeros.
	
\begin{dfn}
	A \textbf{simplicial homotopy} between simplicial maps $f, g \colon  X \to Y$ is a simplicial map $H\colon X \times \Delta^1 \to Y$ such that for every $k$-simplex $\sigma$ of~$X$, we have $H(\sigma, \tau_{k+1}^{k}) = f(\sigma)$ and $H(\sigma, \tau_{0}^{k}) = g(\sigma)$.
\end{dfn}

If $Y$~is a simplicial subset of a free simplicial set, then a simplicial homotopy~$H$ is necessarily given on each $k$-simplex $\sigma$ of~$X$ with vertices $x_0, \ldots, x_k$ by 
\[(\sigma, \tau_i^k) \mapsto (f(x_0), \dots, f(x_{i-1}), g(x_{i}), \dots, g(x_k)).\]
In particular, $H$~is determined by $f$~and~$g$.
Moreover,  such an assignment defines a homotopy (with codomain~$Y$) precisely if $k$-simplex as on the right hand side lies in~$Y$.

The geometric realization~$|H|$ is of course a homotopy from~$|f|$ to~$|g|$, so it induces isomorphisms on homology and homotopy groups in all dimensions.

\subsubsection{Essential connectedness and essential acyclicity}

A directed system of pointed sets $(S_\alpha)_{\alpha \in A}$ is \textbf{essentially trivial} if for each $\alpha\in A$ there is $\beta \ge \alpha$ such that the map $S_\alpha \to S_\beta$ is trivial. The pointed sets to which we intend to apply this notion will be homotopy groups and reduced homology $R$-modules.

Given a simplicial set~$Y$, a \textbf{filtration} of~$Y$ over a directed poset~$A$ is a family of simplicial subsets $Y_\alpha\subseteq Y$ indexed by~$A$, such that $\bigcup_{\alpha \in A}Y_\alpha =Y$ and whenever $\alpha \le \beta$, we have $Y_\alpha \subseteq Y_{\beta}$. For our purposes, the poset~$A$ will most often be the poset of compact subspaces of a topological space.

If a vertex $y\in Y^{(0)}$ is contained in all stages of a filtration $(Y_\alpha)_{\alpha \in A}$, we consider, for each $k\in \NN$, the directed system of homotopy groups $(\pi_k(Y_\alpha, y))_{\alpha \in A}$, which comes with inclusion-induced maps $\pi_k(Y_\alpha, y) \to \pi_k(Y_\beta, y)$ whenever $\alpha\le\beta$.
When $Y$ is nonempty but it is not the case that some $y\in Y^{(0)}$ lies in all the $Y_\alpha$, we may consider the subposet $A_y := \{\alpha \in A \mid y \in Y_\alpha \}$, and the directed system of homotopy groups $(\pi_k(Y_\alpha,y))_{\alpha \in A_y}$, with the inclusion-induced maps. It turns out that if $Y$~is connected, then whether this directed system is essentially trivial does not depend on the basepoint~$y$ \cite[Lemma~2.8]{BHQ24a}. Hence, given $n\in \NN$, we may define the filtration $(Y_\alpha)_{\alpha\in A}$ of~$Y$ to be \textbf{essentially $n$-connected} if for every $k\le n$ and some (hence every) $y\in Y$, the directed system $(\pi_k(Y,y))_{\alpha \in A_y}$ is essentially trivial. For all filtrations we consider in  this article, the simplicial set~$Y$ is connected, so we will routinely suppress basepoints from the notation.

Similarly, for each commutative ring~$R$ and $k\in \NN$, we have a directed system of reduced homology $R$-modules $(\redH_k(Y_\alpha);R)_{\alpha \in A}$, with the inclusion-induced maps.
When discussing homology, there are no basepoint issues to worry about: we say that the filtration $(Y_\alpha)_{\alpha \in A}$ is \textbf{essentially $n$-acyclic over~$R$} if for every $k\le n$ the directed system $(\redH_k(Y_\alpha;R))_{\alpha \in A}$ is essentially trivial.

\subsubsection{Simplicial homotopy equivalences}

\begin{dfn}\label{dfn:hoequiv}
	Given filtrations $(X_\alpha)_{\alpha\in A}$, $(Y_\beta)_{\beta \in B}$ of simplicial sets~$X, Y$ respectively and a map of posets $\phi \colon A \to B$, we say that a family of simplicial maps $(f_\alpha \colon X_\alpha \to Y_{\phi(\alpha)})_{\alpha \in A}$ \textbf{eventually commutes up to simplicial homotopy} if for every $\alpha, \alpha' \in A$ with $\alpha \le \alpha'$, there is $\beta\ge \phi(\alpha')$ such that the following diagram commutes up to simplicial homotopy:
	$$\begin{tikzcd}
		X_\alpha  \ar[d,hook]\ar[r,"f_\alpha"] & Y_{\phi(\alpha)} \ar[r,hook] & Y_\beta \\ X_{\alpha'}  \ar[r,"f_{\alpha'}"] & Y_{\phi(\alpha')} \ar[ur, hook]
	\end{tikzcd}$$
	The word ``eventually'' may be dropped if it is possible to take $\beta = \phi(\alpha')$.
	
	This data $(\phi, (f_\alpha)_{\alpha\in A})$ is a  \textbf{simplicial homotopy equivalence} between the filtrations if there exist $\psi\colon B\to A$ and $(g_\beta\colon Y_\beta \to X_{\psi(\beta)})_{\beta \in B}$ eventually commuting up to simplicial homotopy, such that for every $\alpha \in A$ and $\beta \in B$, 
	there are $\alpha'\ge \psi(\phi(\alpha)), \beta'\ge \phi(\psi(\beta))$ making the following diagrams commute up to simplicial homotopy:	
	$$\begin{tikzcd}
	X_\alpha \ar[r,hook]\ar[d,"f_\alpha"'] & X_{\alpha'}  \\
	Y_{\phi(\alpha)}\ar[r, "g_{\phi(\alpha)}"] & X_{\psi(\phi(\alpha))} \ar[u, hook]
	\end{tikzcd} \qquad
	\begin{tikzcd}
	Y_\beta \ar[r,hook] \ar[d,"g_\beta"'] &Y_{\beta'} \\
	 X_{\psi(\beta)}\ar[r,"f_{\psi(\beta)}"]& Y_{\phi(\psi(\beta))} \ar[u, hook]
	\end{tikzcd}$$
	Then $(\psi, (g_\beta)_{\beta \in B})$ is also a simplicial homotopy equivalence, which we call a \textbf{simplicial homotopy inverse} to~$(\phi, (f_\alpha)_{\alpha \in A})$.
\end{dfn}

For us the relevance of this notion lies in the following fact: 

\begin{lem}[Homotopy equivalences and essential triviality]\label{lem:homotopyequiv}
	If two filtrations $(X_\alpha)_{\alpha\in A}$, $(Y_\beta)_{\beta \in B}$ of connected simplicial sets $X,Y$ are simplicially homotopy-equivalent, then for every commutative ring~$R$ and $n\in \NN$:
    \begin{enumerate}
        \item  $(X_\alpha)_{\alpha\in A}$ is essentially $n$-connected if and only if $(Y_\beta)_{\beta\in B}$~is, and
        \item  $(X_\alpha)_{\alpha\in A}$ is essentially $n$-acyclic over~$R$ if and only if $(Y_\beta)_{\beta\in B}$~is.
    \end{enumerate}
\end{lem}

The first item is a lemma in the previous article \cite[Lemma~2.11]{BHQ24a}, and the second item is established by exactly the same argument.

An important special case is when given two filtrations $(Y_\alpha)_{\alpha\in A}, (Y_\beta)_{\beta\in B}$ of the same simplicial set~$Y$. If there is a poset map $\phi \colon A \to B$ such that for every $\alpha \in A$ we have $Y_\alpha \subseteq Y_{\phi(\alpha)}$, we will say that $(Y_\beta)_{\beta \in B}$ is \textbf{cofinal} to $(Y_\alpha)_{\alpha \in A}$.
If each of the filtrations is cofinal to the other, we will simply say that ``they are cofinal''. It is easy to verify that the inclusion maps then assemble to a pair of mutually inverse simplicial homotopy equivalences. In particular, by Lemma~\ref{lem:homotopyequiv}, cofinal filtrations have the same essential connectedness and essential acyclicity properties.

\subsection{Geometric invariants of locally compact groups} \label{sec:Sigma}

For this section, fix a locally compact Hausdorff topological group~$G$, a commutative ring~$R$, and $n\in \NN$. 

The various $\Sigma$-sets of~$G$ are certain subsets of the $\RR$-vector space $\Homtop(G,\RR)$ of continuous group homomorphisms $\chi \colon G\to\RR$. These homomorphisms are typically called \textbf{characters}, and each defines a topological submonoid $G_\chi := \chi^{-1}([0,+\infty[)$ of~$G$.

Now, the definition of the $\Sigma$-sets uses certain filtrations of the free simplicial set~$\E G$ and the action of the abstract underlying group of~$G$ by left multiplication. Given a character~$\chi$, the relevant filtration is indexed by the poset~$\C(G)$ of compact subspaces of~$G$, and associates to each $C\in\C(G)$ the simplicial subset
\[G_\chi \cdot \E C := \bigcup_{g\in G_\chi } g\cdot \E C.\]
One then defines, respectively, the homotopical \cite[Definition~3.2]{BHQ24a} and homological \cite[Definition~5.1]{BHQ24b} $\Sigma$-sets as follows:

\begin{dfn}\label{dfn:sigma}    
    Membership of a character $\chi\in \Homtop(G,\RR)$ in the $n$-th \textbf{homotopical $\Sigma$-set}~$\TopS^n(G)$ and \textbf{homological $\Sigma$-set}~$\TopS^n(G;R)$ are defined, respectively, by\footnote{The definition of the homological $\Sigma$-sets $\TopS^n(G;R)$ was introduced in a somewhat different setup, where a different filtration of~$\E G$ is used. The translation to the phrasing presented here might not be immediately obvious, so we explain it in detail in Appendix~\ref{sec:appendix}.}:
    \begin{align*}
        \chi \in \TopS^n(G) &\iff \text{$(G_\chi\cdot\E C)_{C\in\C(G)}$ is essentially $(n-1)$-connected},\\
        \chi \in \TopS^n(G;R) &\iff \text{$(G_\chi\cdot\E C)_{C\in\C(G)}$ is essentially $(n-1)$-acyclic over~$R$}.
    \end{align*}
\end{dfn}

We interpret ``being essentially $(-1)$-connected/acyclic'' as a vacuous condition, so the $0$-th $\Sigma$-sets both equal $\Homtop(G;\RR)$. Thus we have descending sequences
\begin{align*}
    \Homtop(G,\RR) & =\TopS^0(G)\supseteq \TopS^1(G) \supseteq\TopS^2(G) \supseteq\ldots \supseteq \TopS^\infty(G)\\
    \Homtop(G,\RR) & =\TopS^0(G;R)\supseteq \TopS^1(G;R) \supseteq\TopS^2(G;R) \supseteq\ldots \supseteq \TopS^\infty(G;R),
\end{align*}
where 
\[\TopS^\infty(G) := \bigcap_{n\in\NN} \TopS^n(G) \qquad \text{and} \qquad \TopS^\infty(G;R) := \bigcap_{n\in\NN} \TopS^n(G;R).\]

From the definition, it is clear that if a character $\chi$~lies in a given $\Sigma$-set, then so do all its positive scalar multiples. In the special case of the $0$-character, where one considers the filtration $(G\cdot \E C)_{C\in\C(G)}$, we recover verbatim the compactness properties $\mathrm C_n$ and~$\mathrm {CP}_n(R)$ introduced by Abels and Tiemeyer~\cite{AT97}. The fact that $\Sigma$-sets refine compactness properties is made more precise by the following statement \cite[Proposition~3.5]{BHQ24a} \cite[Proposition~5.5]{BHQ24b}:

\begin{prop}[The zero character] The following equivalences hold\footnote{The second item is only stated in the given reference for the case $R=\ZZ$, and the proof therein does use this assumption to establish the case~$n=1$. However, the argument supplied for $n\ge 2$ also fits the case~$n=1$, and it in no way relies on the choice of~$R$.}:
    \begin{enumerate}
        \item $\text{$G$~is of type~$\mathrm{C}_n$} \iff 0\in \TopS^n(G) \iff \TopS^n(G) \neq \emptyset$,
        \item $\text{$G$~is of type~$\mathrm{CP}_n(R)$} \iff 0\in \TopS^n(G;R) \iff \TopS^n(G;R) \neq \emptyset$.
    \end{enumerate}
\end{prop}

A more flexible description of the $\Sigma$-sets is given by using a proper action on a locally compact space~$X$, and then considering a filtration of~$\E X$ indexed by the compact subspaces of~$X$ \cite[Proposition~7.2]{BHQ24a}:

\begin{prop}[$\Sigma$-sets via proper actions]\label{prop:proper_actions}
    Let $X$~be a nonempty locally compact Hausdorff space with a proper $G$-action and let $\chi \colon G \to \RR$ be a character. Then:
    \begin{enumerate}
        \item $\chi \in \TopS^n(G)$ if and only if the filtration $(G_\chi \cdot \E K)_{K\in \C(X)}$ is essentially $(n-1)$-connected,
        \item $\chi \in \TopS^n(G;R)$ if and only if the filtration $(G_\chi \cdot \E K)_{K\in \C(X)}$ is essentially $(n-1)$-acyclic over~$R$.\footnote{The given reference only states the result for the homotopical $\Sigma$-sets, but since that proof shows that the filtrations $(G_\chi \cdot \E C)_{C\in \C(G)}$ of~$\E G$ and $(G_\chi \cdot \E K)_{K\in \C(X)}$ of~$\E X$ are simplicially homotopy equivalent, it serves just as well for establishing the homological statement (using Lemma~\ref{lem:homotopyequiv}).}
    \end{enumerate}
\end{prop}

We will in fact only need the ``$\chi=0$'' case of these statements, which concerns the compactness properties $\mathrm C_n$ and $\mathrm{CP}_n(R)$ of Abels and Tiemeyer, and was already established in their article \cite[Theorem~3.2.2]{AT97}.

In historically retrograde fashion, we end this section with a word on the classical $\Sigma$-sets for abstract groups~$\Gamma$. The following is a convenient description, which we will make use of in Section~\ref{sec:BS}.
Given $n\in \NN$, let $X$~be a contractible (resp. $R$-acyclic) CW-complex where $\Gamma$~acts freely by cell-permuting homeomorphisms, and the $n$-skeleton $X^{(n)}$ has only finitely-many $\Gamma$-orbits of cells. Assume moreover that $X^{(0)}$~has exactly one $\Gamma$-orbit, so we may identify $X^{(0)} = \Gamma$. The existence of such~$X$ is equivalent to $\Gamma$~being of type $\mathrm F_n$ (resp. type $\mathrm {FP}_n(R)$).

Now, given a character $\chi\colon \Gamma \to \RR$, and $r\in \RR$, we write~$X_{\chi \ge r}$ to denote the maximal subcomplex of~$X$ containing the vertices in $\chi^{-1}([r, +\infty[)$.

\begin{dfn}
    The character $\chi$ lies in $\Sigma^n(\Gamma)$ (resp. $\Sigma^n(\Gamma;R)$) if and only if the filtration $(X_{\chi \ge r})_{r\in \RR}$ of~$X$ is essentially $(n-1)$-connected (resp. essentially $(n-1)$-acyclic over~$R$).
\end{dfn}

This definition matches the definition of the $\Sigma$-sets for locally compact groups, once $\Gamma$~is equipped with the discrete topology \cite[Theorem~A]{BHQ24a} \cite[Corollary~7.7]{BHQ24b}:

\begin{prop}[$\Sigma$-sets for discrete groups]
If $\Gamma$~is interpreted as a topological group with the discrete topology, then
\[\TopS^n(\Gamma) = \Sigma^n(\Gamma) \qquad \text{and} \qquad \TopS^n(\Gamma;R) = \Sigma^n(\Gamma;R).\]
\end{prop}

\subsection{Commensurability and TDLC completions}

\subsubsection{TDLC completions of Hecke pairs}

Two  subgroups $\Lambda,\Lambda'$ of a group~$\Gamma$ are \textbf{commensurate} if their intersection $\Lambda\cap \Lambda'$ has finite index in both $\Lambda$ and~$\Lambda'$. This is well-known to be an equivalence relation among the subgroups of~$\Gamma$. 

The subgroup~$\Lambda$ is said to be \textbf{commensurated} in~$\Gamma$ 
if it is commensurate to all its conjugates.\footnote{In the literature, commensurated subgroups  are sometimes referred to as ``inert'', ``quasi-normal'' or ``almost normal''.}

\begin{lem}[Examples of commensurated subgroups]
    The following subgroups are always commensurated:
    \begin{enumerate}
        \item normal subgroups,
        \item compact open subgroups of a topological group,
        \item preimages of commensurated subgroups by group homomorphisms.
    \end{enumerate}
\end{lem}
\begin{proof}
(1) It is obvious that normal subgroups are commensurated.

(2) If $L$~is a compact open subgroup of a topological group~$G$, then for every $g\in G$ the cosets of $L \cap g L g^{-1}$ contained in~$L$ partition it into open subsets. Compactness of~$L$ implies they must be finite in number (and similarly within~$gLg^{-1}$). 

(3) Let $\phi \colon \Gamma \to G$ be a group homomorphism, let $L\subseteq G$ be a commensurated subgroup and let $\Lambda:= \phi^{-1}(L)$, which we will show is commensurated in~$\Gamma$.

First note that the $\phi$-induced map of cosets $\Gamma / \Lambda \to G/L$ is injective, since for every $g\in \Gamma$ we have
$$\phi^{-1}(\phi(g\Lambda)) = \phi^{-1} (\phi(g) L)  =\phi^{-1}(\phi(g)) \phi^{-1}(L)=g \ker (\phi) \Lambda = g\Lambda.$$

Now, given $g\in \Gamma$,
let $\{g_i\}_{i\in I} \subseteq \Gamma$ be representatives of all the left $\Lambda$-cosets having nonempty intersection with~$g \Lambda g^{-1}$, and assume all~$g_i\Lambda$ are distinct. In particular, $[g\Lambda g^{–1} : \Lambda \cap g \Lambda g^{-1}] = |I|$. By the preceding observation, the cosets $\{\phi(g_i) L\}_{i\in I}$ are all distinct. Moreover, they cover~$L$ and intersect it nontrivially, so $[\phi(g)L \phi(g)^{-1}: \phi(g)L \phi(g)^{-1} \cap L] =|I|$. Since $L$~is commensurated, $I$~is finite.

By the same argument, also the index $[\Lambda: \Lambda \cap g \Lambda g^{-1}] = [g^{-1}\Lambda g: \Lambda \cap g^{-1} \Lambda g]$ is finite, so we conclude $\Lambda$~is commensurated.
\end{proof}

A {\bf Hecke pair} is a pair of groups $(\Gamma, \Lambda)$, where $\Gamma$ is a Hausdorff topological group and $\Lambda$ is an open
subgroup of $\Gamma$ that is commensurated. 
A {\bf TDLC completion} of $(\Gamma, \Lambda)$ is a TDLC group~$G$ together with a continuous homomorphism $\phi\colon \Gamma\to G$ with dense image such that $\Lambda=\phi^{-1}(L)$ for some compact open subgroup $L\subseteq G$.  Note that the map~$\phi$ does not have to be injective. A number of instances of TDLC~completions have been studied in the literature~\cite{S80, B93}, culminating in a unified theory formulated by Reid and Wesolek~\cite{RW19}. 

Given a Hecke pair $(\Gamma,\Lambda)$, two TDLC completions play a central role in the theory: the  Schlichting completion $\alpha\colon \Gamma\to \Gamma\slcpt \Lambda$, which is one of the protagonists of this article and whose explicit construction will be recalled in Section~\ref{sec:schlichting}, and the  Belyaev completion $\beta \colon \Gamma\to\hat \Gamma_\Lambda$.
Reid and Wesolek have shown \cite[Theorem~5.4]{RW19}\footnote{The given reference states the existence of unique continuous quotient maps~$\hat \phi, \check \phi$ with compact kernels, which is a slightly less refined statement. But uniqueness among continuous maps follows from the image of $\Gamma$~being dense in all completions.} that these are, respectively, the final and initial completions, in the following sense:
\begin{thm}[Universal TDLC completions]
\label{thm:compl}
Let $\phi\colon \Gamma \to G$ be some TDLC completion of a Hecke pair~$(\Gamma, \Lambda)$.  Then there are unique continuous maps $\hat\phi \colon \hat \Gamma_{\Lambda}
 \to G$ and $\check\phi  \colon G \to \Gamma\slcpt \Lambda$ such that the following diagram commutes: 
\[
\xymatrix{
& \Gamma\ar[ld]_{\beta}\ar[d]^\phi\ar[rd]^{\alpha}\\
\hat \Gamma_{\Lambda}\ar[r]_{\hat\phi}& G\ar[r]_{\check\phi}&\Gamma\slcpt \Lambda}.
\]
Moreover, $\hat\phi$ and $\check\phi$ are quotient maps with compact
kernels. 
\end{thm}
The fact that all TDLC~completions of~$(\Gamma,\Lambda)$ differ only by a compact normal subgroup implies that the compactness properties of the TDLC completions depend only on the pair $(\Gamma,\Lambda)$ \cite[Theorem~3.20]{CCC20}. In Proposition~\ref{prop:completions_have_same_Sigma}, we will see that a similar statement holds, more generally, for $\Sigma$-sets.

\subsubsection{The Schlichting completion}\label{sec:schlichting}

Here we explain a construction of the Schlichting completion~$\Gamma \slcpt \Lambda$ of a Hecke pair $(\Gamma, \Lambda)$, which was first introduced by Tzanev~\cite{T03} following an idea of Schlichting~\cite{S80}. 

The action of $\Gamma$ on the left coset space $\Gamma/\Lambda$ defines the continuous homomorphism $$\alpha\colon \Gamma\to\mathrm{Sym}(\Gamma/\Lambda),$$
where $\mathrm{Sym}(\Gamma/\Lambda)$ carries the topology of pointwise convergence. As $\Lambda$ is open in~$\Gamma$, the space $\Gamma /\Lambda$ is discrete, so this is also the compact-open topology.

\begin{dfn}
The  {\bf Schlichting completion} $\Gamma\slcpt \Lambda$ of~$(\Gamma,\Lambda)$ is the closure~$\overline{\alpha(\Gamma)}$ in $\mathrm{Sym}(\Gamma/\Lambda)$, together with the continuous homomorphism $\alpha\colon \Gamma\to\Gamma\slcpt\Lambda$. 
\end{dfn}

Evaluation at~$\Lambda$ defines a function $\ev\colon \Gamma\slcpt\Lambda\ \to \Gamma / \Lambda$ to the space of left cosets. Since the topology on $\Gamma \slcpt \Lambda$~is induced from the compact-open topology on~$\Sym(\Gamma/\Lambda)$, it follows that $\ev$~is continuous (in fact, as $\Gamma / \Lambda$~is discrete, $\ev$~is locally constant). We thus have a commutative triangle of topological $\Gamma$-spaces
$$\begin{tikzcd}
    \Gamma \ar[d,"p"']\ar[r, "\alpha"]&\Gamma\slcpt\Lambda\ar[dl,"\ev"]\\
    \Gamma / \Lambda
\end{tikzcd},$$
where $p$~is the canonical projection.

Given any TLDC completion, one can give a convenient description of the Shclichting completion using the following criterion:

\begin{prop}[Schlichting completion from any TDLC completion]\label{prop:schlichting_from_compl}
    Let $\phi \colon \Gamma \to G$ be a TDLC completion of the Hecke pair~$(\Gamma, \Lambda)$, with $L:=\overline{\phi(\Lambda)}$.
    
    Then the kernel of the quotient map $\check \phi \colon G \to \Gamma \slcpt \Lambda$ (as in Theorem~\ref{thm:compl}) is the normal core $N:=\operatorname{Core}_G(L)=\bigcap_{g\in G} gLg^{-1}$. In particular, $\Gamma\slcpt \Lambda \cong G / N$.
\end{prop}

\begin{proof}
The proof is contained in the commutative diagram
\[\begin{tikzcd}
    \Gamma \ar[d,"\alpha"'] \ar[r, "\phi"] & G \ar[d,"\alpha"] \ar[dr,->>,"q"]  \ar[dl, "\check \phi"', ->>]\\
    \Gamma \slcpt \Lambda\ar[r,"\cong"']&G\slcpt L&G / N \ar[l,->>,"\check q"]
\end{tikzcd},\]
which we now explain.

By definition of~$\check \phi$, the upper left triangle commutes.
The isomorphism on the bottom left is given by a result of Shalom and Willis \cite[Lemma~3.6]{SW13}\footnote{The given reference only states the existence of this isomorphism for $\Gamma$~discrete. But this assumption is not used in the proof.}. Inspecting their proof reveals that this isomorphism is induced by~$\phi$, in the sense that a permutation in $\Gamma\slcpt\Lambda$ given by left multiplication with an element $g \in \Gamma$ is mapped to the permutation in~$G\slcpt L$ given by left multiplication with~$\phi(g)$.
In other words, the square on the left is commutative. But by uniqueness of~$\check \phi$, the composition $G\xrightarrow{\alpha} G\slcpt L \xleftarrow{\cong} \Gamma\slcpt\Lambda$ equals~$\check\phi$, so the bottom left triangle also commutes.

We now turn our eyes to the right side of the diagram. For the canonical projection $q\colon G\to G/N$, the image~$q(L)$ is compact and open (as $q$~is an open continuous map), and we have $q^{-1}(q(L)) = LN =L$. Thus $q$~is a TDLC completion of~$(G,L)$, which yields the induced quotient map~$\check q$.
A straightforward verification shows that $N$~is the kernel of $\alpha\colon G \to G\slcpt L$, and so $\check q$~is injective, and thus an isomorphism.

The commutative triangle formed by $G, \Gamma\slcpt \Lambda$ and~$G/N$ establishes the statment.
\end{proof}

\section{Almost commutativity with compacts}\label{sec:almost_comm}

In this section, we introduce and discuss the key property enjoyed by open commensurated subgroups $\Lambda \subseteq \Gamma$ that allows us to sidestep the assumption of normality in earlier results \cite[Theorem~8.2]{BHQ24a} \cite[Theorem~J]{BHQ24b} analogous to our Theorem~\ref{thm:main_intro}.

\begin{dfn}\label{dfn:almost_comm}
    A closed subgroup $\Lambda$ of a locally compact Hausdorff group~$\Gamma$ is said to {\bf almost commute with compacts} if for every compact $C\subseteq \Gamma$, there is a compact $C'\subseteq \Gamma$ such that $C\Lambda \subseteq \Lambda C'$ and $\Lambda C\subseteq C'\Lambda$. 
\end{dfn}

\begin{rem}[Left vs. right commutativity]
    To check that $\Lambda$~almost commutes with compacts, it suffices to show that for every compact~$C$, there is a compact~$C'$ with $C\Lambda \subseteq \Lambda C'$. This property applied to $C^{-1}$ yields $C''$ with $\Lambda C \subseteq (C'')^{-1}\Lambda$, and then the union $C'\cup (C'')^{-1}$ is as required by the definition.
\end{rem}

If $\Lambda\subseteq \Gamma$ is a closed normal subgroup, then it is obvious that $\Lambda$~almost commutes with compacts. Another example, which is the most relevant for this article, is the case where $\Lambda$~is an open commensurated subgroup:

\begin{prop}[Almost commutativity of compact open subgroups]\label{prop:comm_exchange}
  For every Hecke pair $(\Gamma, \Lambda)$, the subgroup $\Lambda$~almost commutes with compacts.
\end{prop}

\begin{proof} We will prove that, in fact, for every $C\in \C(G)$ there exists a finite $F\subseteq G$ satisfying $C\Lambda\subseteq \Lambda F$. 

    Since $\Lambda$ is open and $C$~is compact, we see that $C$~is contained in the union of finitely many left $\Lambda$-cosets, that is, there is a finite subset $C_0\subseteq C$ such that $C\subseteq C_0\Lambda$, and hence $C\Lambda = C_0\Lambda$. We may therefore assume from now on that $C$~is finite.

    For each $g\in C$, choose a (finite) set $F_g$ of right coset representatives of $\Lambda \cap g^{-1} \Lambda g$ in~$\Lambda$, so $$\Lambda = (\Lambda \cap g^{-1} \Lambda g)F_g \subseteq g^{-1} \Lambda gF_g.$$
    Thus, we have $g\Lambda \subseteq \Lambda gF_g$, and defining $F:= \bigcup_{g\in C} g F_g$ yields
    $$C\Lambda= \bigcup_{g\in C}g\Lambda \subseteq \bigcup_{g\in C} \Lambda g F_g = \Lambda \bigcup_{g\in C}  g F_g = \Lambda F.\qedhere$$
\end{proof}

It is unclear whether it is possible to relax the hypothesis on~$\Lambda$ in  Proposition~\ref{prop:comm_exchange} from ``open commensurated'' to ``closed commensurated''.

\begin{ques}\label{question}
Let $\Gamma$ be a locally compact Hausdorff group.
Does every closed commensurated subgroup of~$\Gamma$ almost commute with compacts?
\end{ques}

As communicated to us by George Willis, the answer is affirmative in the case where $\Gamma$~is a TDLC~group. The following proof was kindly supplied by him \cite{Wil25}.
\begin{prop}[Almost commutativity in TDLC groups]\label{prop:Willis} 
Let $\Lambda$ be a closed commensurated subgroup of a TDLC group~$\Gamma$. Then $\Lambda$ almost  commutes with compacts.
\end{prop}

\begin{proof}
By the argument used in the proof of Proposition~\ref{prop:comm_exchange}, $\Lambda$ almost commutes with finite sets (that is, for every finite set $F\subseteq \Gamma$ there is a finite set $F'\subseteq\Gamma$ such that $F\Lambda\subseteq \Lambda F'$ and $\Lambda F\subseteq F'\Lambda$). Suppose that $C\subseteq\Gamma$ is compact and that $U$ is a compact open subgroup of $\Gamma$. Then $C\subseteq UF$ with $F$ a finite set and $C\Lambda\subseteq U\Lambda F'$ with $F'$~a finite set. Hence, if $U\Lambda\subseteq \Lambda K$ with $K$ compact, then $C\Lambda\subseteq \Lambda(KF')$ with $KF'$ compact and the claim holds for~$C$. We thus set out to find such $U$ and~$K$.

 Define the function
\begin{align*}
    \nu \colon \Gamma &\to \NN\\
    x & \mapsto [x\Lambda x^{-1} : \Lambda \cap x\Lambda x^{-1}].
\end{align*}
It is shown next that $\nu$~is lower semicontinuous. Since $\NN$~is discrete, this is equivalent to the statement that every $x\in \Gamma$ has a neighbourhood $\mathcal{U}\ni x$ such that for all $y\in\mathcal{U}$ we have $\nu(y)\geq \nu(x)$. For this, first note that, as $\Lambda$~is commensurated, there are $k_1,\cdots,k_n\in x\Lambda x^{-1}$ with $$x\Lambda x^{-1}\subseteq \bigcup_{\ell=1}^n \Lambda k_{\ell},$$
and we may assume all cosets $\Lambda k_\ell$ are distinct.
Set $b_{\ell} = x^{-1}k_\ell x\in \Lambda$ for each~$\ell$, and choose a compact open subgroup $V\subseteq \Gamma$ such that the double cosets $\Lambda k_\ell V$ are distinct open neighbourhoods of~$k_\ell$, $\ell=1,\dots,n$ (such~$V$ exists because each coset $\Lambda k_\ell$ is closed and equal to the set $\bigcap\left\{ \Lambda k_\ell V\mid V\leq \Gamma\text{ is compact and open}\right\}$).
Then there is a neighbourhood $\mathcal{U}\ni x$ such that for every $\ell$ and every $y\in \mathcal U$, we have $yb_\ell  y^{-1}\in \Lambda(xb_\ell x^{-1})V = \Lambda k_\ell V$. Hence $\Lambda(yb_\ell  y^{-1})V$ are distinct double cosets for each $y\in\mathcal{U}$ and, in particular, $\Lambda(yb_\ell y^{-1})$ are distinct $\Lambda$-cosets intersecting $y\Lambda y^{-1}$ at $yb_\ell y^{-1}$. 
Therefore $\nu(y)\geq\nu(x)$ for every $y\in\mathcal{U}$ as claimed.

The preceding argument implies that
$$
A_n := \left\{x\in \Gamma\mid \nu(x)<n\right\}
$$
is a closed subset of $\Gamma$ for each $n$. We also have that $\Gamma = \bigcup_{n\geq1} A_n$ because $\Lambda$ is commensurated by $\Gamma$. Hence, by the Baire Category Theorem, there is~$n$ such that $A_n$ has non-empty interior. Choose $x\in \Gamma$ and a compact open subgroup $U\subseteq \Gamma$ such that $xU\subseteq A_n$. Let $g_1, \ldots,g_{\nu(x^{-1})}\in x^{-1}\Lambda x$ be such that $x^{-1}\Lambda x \subseteq \bigcup_{i=1}^{\nu(x^{-1})} \Lambda g_i$.
Moreover, since for every $u\in U$ we have $\nu(xu) = [u\Lambda u^{-1} : x^{-1} \Lambda x \cap u\Lambda u^{-1}]$, there are $h_1, \ldots,h_{\nu(xu)}\in u\Lambda u^{-1}$ (depending on~$u$) that satisfy
\[u\Lambda u^{-1} \subseteq \bigcup_{j=1}^{\nu(xu)} (x^{-1}\Lambda x) h_j \subseteq  \bigcup_{j=1}^{\nu(xu)} \bigcup_{i=1}^{\nu(x^{-1})}\Lambda g_i h_j.\]
Therefore, $\nu(u) \le \nu(xu)\nu(x^{-1}) < n \nu(x^{-1})$.

Since $\left\{[u\Lambda u^{-1} : \Lambda\cap u\Lambda u^{-1}]\mid u\in U\right\}$ is bounded, the Bergman-Lenstra Theorem \cite[Theorem 6(iii)]{BLW89} shows that there is  a subgroup $\Lambda'\subseteq \Gamma$ commensurable with $\Lambda$ and normalised by~$U$.
In particular, there are $l_1,\ldots,l_m\in \Lambda$ and $l_1',\ldots,l_{m'}'\in \Lambda'$ such that
$$
\Lambda\subseteq \bigcup_{i=1}^m \Lambda'l_i \text{ and }\Lambda'\subseteq \bigcup_{j=1}^{m'} \Lambda l'_j.
$$

Then 
$$
U\Lambda \subseteq U\bigcup_{i=1}^m \Lambda'l_i = \bigcup_{i=1}^m \Lambda'U l_i \subseteq \bigcup_{i=1}^m \bigcup_{j=1}^{m'}\Lambda l'_jU l_i=\Lambda K
$$
with $K = \bigcup_{i=1}^m \bigcup_{j=1}^{m'} l'_jU l_i$ compact, as required.
\end{proof}

\section{Character spaces of  TDLC completions}\label{sec:char_space}

For this section, fix a Hecke pair~$(\Gamma, \Lambda)$ and let $\phi\colon \Gamma \to G$ be a TDLC completion of~$(\Gamma, \Lambda)$ with $L:= \overline{\phi(\Lambda)}$. We notate the Schlichting  completion as $\alpha \colon \Gamma \to \Gamma\slcpt\Lambda=: \G$, and write $\L := \overline {\alpha(\Lambda)}$.
Finally, fix a locally compact Hausdorff group~$H$. We denote by $\Homtop(\Gamma,H)_\Lambda$ the set of continuous homomorphisms $\Gamma \to H$ whose kernel contains~$\Lambda$ (and similarly with other Hecke pairs in place of $(\Gamma, \Lambda)$).

\begin{prop}[Maps out of TDLC completions]\label{prop:maps_out_of_completion}
    Precomposition with~$\phi$ induces a bijection
    $$\phi^*\colon \Homtop(G, H)_L \xrightarrow{\cong} \Homtop(\Gamma, H)_\Lambda.$$
\end{prop}
\begin{proof}
    The map $\phi^*$ is injective because $\phi$~has dense image.

    To see that $\phi^*$~is surjective, observe that the right-hand triangle in Theorem~\ref{thm:compl} yields a commutative triangle
    $$\begin{tikzcd}
        \Homtop(\Gamma, H)_\Lambda &\\
        \Homtop(G,H)_L \ar[u,"\phi^*"]& \Homtop(\G, H)_{\L}\ar[ul,"\alpha^*"']\ar[l,"\check\phi^*"]
    \end{tikzcd},$$
    where to justify the subscripts, one uses the fact that $\phi$~and~$\alpha$ map $\Lambda$ densely into~$L$ and~$\L$, respectively.
    This diagram reduces us to showing that $\alpha^*$~is surjective, which is the content of Lemma~\ref{lem:mapfromslcpt} below.
\end{proof}

\begin{lem}[Induced maps on the Schlichting completion]\label{lem:mapfromslcpt}
    For every $f \in\Homtop(\Gamma,H)_\Lambda$ the assignment
    \begin{align*}
        \tilde f\colon \G &\to H\\
        \sigma &\mapsto f(g),
    \end{align*}
    where $g\in \Gamma$ is any representative of the coset $\sigma(\Lambda)$, defines a locally constant homomorphism that is trivial on~$\L$.
\end{lem}
\begin{proof}
    Since $f$~vanishes on~$\Lambda$, it is clear that the assignment is independent of the choice of~$g$.

    To verify that $\tilde f$ is a group homomorphism,    
    let $\sigma, \sigma' \in \G$.
    Choose a representative $g\in \Gamma$ of~$\sigma(\Lambda)$, so $\tilde f(\sigma) = f(g)$.
    Since $\Sym(\Gamma/\Lambda)$ has the topology of pointwise convergence and $\Gamma / \Lambda$~is discrete, the action of~$\sigma'$ on any finite subset of $\Gamma / \Lambda$ is realized by acting with some element of~$\Gamma$. In particular, there is $g'\in \Gamma$ such that
    $$g'\Lambda = \sigma'(\Lambda)\qquad  \text{and}\qquad g'\sigma(\Lambda) = \sigma'(\sigma(\Lambda)).$$
    From the the first condition, we get $\tilde f(\sigma') = f(g')$, and the second condition yields
    $\sigma'(\sigma (\Lambda)) = g'g\Lambda$, whence
    $$\tilde f (\sigma'\circ \sigma) = f(g'g) = f(g')f(g) = \tilde f(\sigma')\tilde f(\sigma).$$

    We now show $\tilde f$~is locally constant at each permutation~$\sigma\in \G$. Let $g\in \Gamma$ represent the coset $\sigma(\Lambda)$ and consider the open neighborhood~$U$ of~$\sigma$ given by
    $$U:= \{\xi\in \G \mid \xi(\Lambda) = g\Lambda\}.$$
    Then, for every $\xi \in U$, we have $\tilde f(\xi) = f(g)$.

    Finally, for every~$h\in\Lambda$, we have $\tilde f(\alpha(h)) = f(1)=1$. Since $H$~is Hausdorff, it follows from continuity that $\tilde f$~vanishes on the closure $\overline{\alpha(\Lambda)}=\L$.
\end{proof}

\begin{cor}[Character spaces of TDLC completions]\label{cor:char_space} Precomposition with $\phi$ induces an isomorphism of $\RR$-vector spaces
   $$\phi^*\colon \Homtop(G, \RR) \xrightarrow{\cong} \Homtop(\Gamma, \RR)_\Lambda.$$
\end{cor}
\begin{proof}
    This is a direct consequence of Proposition~\ref{prop:maps_out_of_completion} and the fact that, since $L$~is compact, all characters on~$G$ vanish on~$L$.
 \end{proof}

Note in particular that the map~$\check\phi$ from Theorem~\ref{thm:compl} induces an isomorphism of character spaces
\[ \check\phi^* \colon \Homtop(\G,\RR) \to \Homtop(G,\RR).\]

\begin{prop}[All TDLC completions have the same $\Sigma$-theory]\label{prop:completions_have_same_Sigma}
    For every $n\in \NN$ and commutative ring~$R$,
    the isomorphism~$\check\phi^*$ restricts to bijections $\TopS^n(\G) \cong\TopS^n(G)$ and $\TopS^n(\G;R) \cong\TopS^n(G;R)$.

    As a consequence, the $n$-th homotopical and homological $\Sigma$-sets of any two TDLC completions of~$(\Gamma, \Lambda)$ are in canonical bijection.
\end{prop}
\begin{proof}
    Theorem~\ref{thm:compl} tells us that $\check \phi \colon G\to\G$ is a quotient map with compact kernel~$N$. In particular, $N$~is of type~$\mathrm{C}_n$ and~$\mathrm{CP}_n(R)$. Thus, by known results \cite[Theorem~8.2]{BHQ24a}\cite[Theorem~10.4]{BHQ24b}, which we will generalize in Propositions~\ref{prop:sigmaofquotient} and~\ref{prop:sigmaofquotient_homological}, each character on~$\G$ lies in $\TopS^n(\G)$ if and only if its pullback to~$G$ lies in~$\TopS^n(G)$ (and similarly for the homological $\Sigma$-sets)\footnote{The given reference for the homological theorem is only stated for $R=\ZZ$, but, as we will see in the proof of Proposition~\ref{prop:sigmaofquotient_homological}, the argument works on arbitrary~$R$.}. As we already observed that each character in $G$~is the pullback of a unique character in~$\G$, the conclusion follows.
\end{proof}

\section{Proof of Theorem~\ref{thm:main_intro}}\label{sec:main_proof}

We now start making our way to the proof of our main result, Theorem~\ref{thm:main_intro}. Let us fix for all of Section~\ref{sec:main_proof} a locally compact Hausdorff group~$\Gamma$ and a closed subgroup $\Lambda\subseteq \Gamma$ that almost commutes with compacts. Denote by $p\colon \Gamma \onto \Gamma /\Lambda$ the projection to the space of left cosets, which carries an action of~$\Gamma$ by left multiplication. 
Fix also a commutative ring~$R$, a natural $n\in \NN$, and a character $\chi\colon \Gamma \to \RR$ that vanishes on~$\Lambda$.

The proof will consist of two main parts. In Section~\ref{sec:part1}, we will relate membership of~$\chi$ in the $\Sigma$-sets of~$\Gamma$ to the essential connectedness / acyclicity properties of the filtration $(\Gamma_\chi \cdot \E F)_{F \in \C(\Gamma / \Lambda)}$ of $\E (\Gamma / \Lambda)$, under the assumption that $\Lambda$~is of type~$\mathrm{C}_n$ / $\mathrm{CP}_n(R)$. This task will be performed separately for the homological and the homotopical $\Sigma$-sets, and both arguments will closely follow the proofs of analogous theorems established earlier for $\Lambda$~normal \cite[Theorem~8.2]{BHQ24a} \cite[Theorem~J]{BHQ24b}. We will sidestep the fact that $\Lambda$~is not normal using the assumption that $\Lambda$~almost commutes with compacts.

In Section~\ref{sec:part2}, we introduce the assumption that $\Lambda$~is open in~$\Gamma$, so  $(\Gamma,\Lambda)$~is a Hecke pair and one can define the Schlichting completion $\G:= \Gamma\slcpt\Lambda$. We then show that the filtration
$(\Gamma_\chi \cdot \E F)_{F\in \C(\Gamma/\Lambda)}$ of~$\E(\Gamma/ \Lambda)$ is homotopy-equivalent to the filtration of~$\E \G$ that defines membership of the induced character $\tilde \chi\colon \G \to \RR$ in the $\Sigma$-sets of~$\G$. Proposition~\ref{prop:completions_have_same_Sigma} will allow us to extend the result to all other TDLC completions $\phi \colon \Gamma \to G$. This second part of the proof makes no distinction between homotopical and homological $\Sigma$-sets.

\subsection{From $\Gamma$ to $\Gamma / \Lambda$}\label{sec:part1}

As mentioned earlier, throughout all of Section~\ref{sec:part1} we will assume that the closed subgroup $\Lambda\subseteq\Gamma$ almost commutes with compacts.

In comparing the relevant filtrations of $\E \Gamma$ and $\E(\Gamma / \Lambda)$, 
we will use the fact that the poset of compact subsets of $\Gamma / \Lambda$ may be replaced by that of~$\Gamma$ \cite[Corollary~8.4(2)]{BHQ24a}\footnote{The quoted reference presents this statement under the assumption that $\Lambda$~is a normal subgroup of~$\Gamma$, but the same proof applies verbatim in our case.}:

\begin{lem}[Re-indexing by $\C(\Gamma)$]\label{lem:reindexing}
    We have 
    \[\C(\Gamma / \Lambda) = \{p(C) \mid C\in \C(\Gamma)\}.\]
\end{lem}

\subsubsection{Homotopical lifting}\label{sec:homotopical_lifting}
The results in Section~\ref{sec:homotopical_lifting} are adjusted from the article on homological $\Sigma$-sets for locally compact groups \cite[Lemma~8.5 and Theorem~8.2]{BHQ24a}.

We remind the reader that a simplicial set $X$ is \textbf{finite relative to} a simplicial subset~$A$ if $X$~has only finitely many nondegenerate simplices that do not lie in~$A$.

\begin{lem}[Lifting along $\Lambda$~of type~$\mathrm 
C_n$]\label{lem:lift}
    If $\Lambda$ is of type~$\mathrm{C}_n$, then for every $C\in \C(\Gamma)$, there is $D\in \C(\Gamma)_{\supseteq C}$ satisfying the following lifting property: given any pair $A\subseteq X$ of simplicial sets of dimension at most~$n$, with $X$~finite relative to~$A$, and a commutative square
	\[\begin{tikzcd}
		A \arrow[r, "\eta"] \arrow[d, hook]& \Gamma_{\chi} \cdot \E C \arrow[d, "p"]\\
		X \arrow[r,"\mu"] & \Gamma_\chi \cdot \E p(C)
	\end{tikzcd},\]
	there are $m\in \NN$ and a map $\tilde \mu \colon \SD^m(X) \to \Gamma_{\chi}\cdot \E D$ such that in the diagram
	\[\begin{tikzcd}
		\SD^m(A) \ar[d,hook]\ar[r,"\Phi^m"]& A \arrow[r, "\eta"] &  \Gamma_{\chi} \cdot \E C  \arrow[r,hook] &  \Gamma_{\chi} \cdot \E D \arrow[d, "p"] \\
		\SD^m(X) \ar[r,"\Phi^m"']  \arrow[rrru,"\tilde \mu"] & X \arrow[r,"\mu"'] & \Gamma_\chi \cdot \E p(C) \arrow[r,hook] & \Gamma_\chi \cdot \E p(D)
	\end{tikzcd},\]
		the upper triangle commutes and the lower triangle commutes up to simplicial homotopy.
\end{lem}

\begin{proof}
    Most of the proof will consist of finding $D^\circ\in \C(\Gamma)_{\supseteq C}$ satisfying all properties required of~$D$, with the exception that the lower triangle of the diagram might not commute up to homotopy. Afterwards, we will find $D\in \C(\Gamma)_{\supseteq D^\circ}$ such that this additional requirement is met.

    Using the assumption that $\Lambda$~almost commutes with compacts, let 
    $C' \in \C(\Gamma)$ satisfy  $C\Lambda\subseteq \Lambda C'$ and $C^{-1}\Lambda \subseteq \Lambda C'^{-1}$.

    To construct $D^\circ$, we proceed by induction over~$n$. Along the way, we also establish the following auxiliary statement:
    \begin{itemize}
		\item[($\star$)] Let $\sigma$ be a simplex of~$X$, and let $g'\in \Gamma_\chi$ be such that $\mu(\sigma)$ lies in~$g' \cdot  \E p(C)$. Then $\tilde \mu (\SD^m(\sigma))\subseteq g'\Lambda\cdot \E(C'C'^{-1}D^\circ)$.
	\end{itemize}

    For $n = 0$, take $D^\circ:=C'$, and suppose we are given a commutative square as above. For each vertex~$x$ of~$X$, we define $\tilde \mu(x)$ as follows: if $x\in A$, then we have not choice but to set $\tilde \mu(x) = \eta(x)$; if $x\not \in A$, define $\tilde \mu (x)$ to be any $p$-preimage of~$\mu(x)$. Clearly, the relevant triangle commutes. To verify~$(\star)$, fix $x\in X^{(0)}$ and $g' \in \Gamma_\chi$~with $\mu(x) \in g' \cdot \E p(C)$. We have  $\mu(x) \in p(g'C)$, so our construction of~$\tilde \mu$ by taking preimages implies that $\tilde\mu(x) \in g'C\Lambda\subseteq g'\Lambda C'$, which makes $\tilde \mu(x)$ a vertex of $g'\Lambda\cdot \E C' \subseteq g'\Lambda\cdot \E(C'C'^{-1}D^\circ)$.

    Now let $n\ge 1$. By induction, there is $D_-^\circ\in \C(\Gamma)_{\supseteq C}$ satisfying the stated extension property on commutative squares with $X$~of dimension at most~$n-1$. Consider the proper action of~$\Lambda$ on~$\Gamma$ by left multiplication.
    Since $\Lambda$~is of type~$\mathrm C_n$, Proposition~\ref{prop:proper_actions} applied to the zero character on~$\Lambda$ tells us that the filtration $(\Lambda\cdot \E K)_{K \in \C(\Gamma)}$ of $\E \Gamma$ is essentially $(n-1)$-connected. Thus, there exists $D^\circ \in \C(\Gamma)_{\supseteq C'C'^{-1}D_-^\circ}$ such that the inclusion $\Lambda\cdot \E(C'C'^{-1}D_-^\circ) \into \Lambda\cdot \E D^\circ$ is $\pi_{n-1}$-trivial. We claim that such~$D^\circ$ is as desired.

    To verify the extension property, let $X, A, \eta, \mu$ fit in a commutative square as above, with $X$~of dimension at most~$n$. By choice of~$D_-^\circ$, there are $m_-\in \NN$ and $\tilde\mu_- \colon \SD^{m_-}(X^{(n-1)}) \to \Gamma_\chi\cdot \E D_-^\circ$ solving the extension problem given by the restriction of $\eta$ and~$\mu$ to the $(n-1)$-skeleta $A^{(n-1)}\subseteq X^{(n-1)}$.

    We now define~$\tilde \mu$ on the (subdivided) non\-de\-ge\-ne\-rate $n$-simplices~$\sigma$ of~$X$ that are not in~$A$. Given such~$\sigma$, choose $g\in \Gamma_\chi$ such that $\mu(\sigma)$ lies in $g\cdot \E p(C)$.
    Property~$(\star)$ on~$\tilde \mu_-$ tells us that for each face~$\tau$ of~$\sigma$, we have $\tilde \mu_-(\SD^{m_-}(\tau)) \subseteq g\Lambda \cdot \E(C'C'^{-1}D_-^\circ)$.
    In other words, $\tilde \mu_-(\SD^{m_-}(\partial \sigma)) \subseteq g\Lambda \cdot \E(C'C'^{-1}D_-^\circ)$. Therefore, by choice of~$D^\circ$ (and Lemma~\ref{lem:comb_pi_k}), there are $m_\sigma \ge m_-$ and a map $\tilde \mu_\sigma \colon \SD^{m_\sigma}(\sigma) \to g\Lambda\cdot \E D^\circ$ extending~$\tilde \mu_- \circ \Phi^{m_\sigma - m_-}$.

    Since $X$ is finite relative to~$A$, we may take~$m\in \NN$ as the maximum of all the~$m_\sigma$, and we are now ready to construct $\tilde \mu \colon \SD^m(X) \to \Gamma_\chi\cdot \E D^\circ$.
    The definition on~$\SD^m(A)$ is forced to be as $\eta \circ \Phi^m$. On $\SD^m(X^{(n-1)})$, we define~$\tilde \mu$ as $\tilde \mu_- \circ \Phi^{m-m_-}$, which is consistent with the definition on~$\SD^m(A)$.
    Finally, on each nondegenerate $n$-simplex~$\sigma$ of~$X$ we define~$\tilde \mu$ as $\tilde \mu_\sigma \circ \Phi^{m-m_\sigma}$, which again matches the definition on $\SD^m(A \cup X^{(n-1)})$. Thus, $\tilde \mu$~is defined so that the upper triangle of the diagram in the lemma commutes. 

    Next we establish~$(\star)$, so let $\sigma$ and~$g'$ be as in its hypothesis.
	If $\sigma$~lies in~$X^{(n-1)}$, the statement follows directly from the induction hypothesis.
	
	If $\sigma$~is a nondegenerate $n$-simplex of~$X$ that is not in~$A$, and $g\in \Gamma_\chi$ is the element chosen when constructing~$\tilde \mu_\sigma$, we have:
	\begin{enumerate}
		\item $\mu(\sigma) \subseteq g\cdot \E p(C)$,
		\item $\tilde \mu(\SD^m(\sigma)) \subseteq g\Lambda\cdot \E D^\circ$, and 
		\item $\mu(\sigma) \subseteq g'\cdot \E p(C)$, the assumption of~$(\star)$.
	\end{enumerate}

	Conditions 1 and 3 imply that $p(gC) \cap p(g'C) \neq \emptyset$, which means there are $x, x'\in C$ and $h\in \Lambda$ with $gx = g'x'h$. Therefore $g = g' x' h x^{-1}\in g'C \Lambda C^{-1}$, and so Condition~2 allows us to conclude:
\begin{align*}
    \tilde\mu(\SD^m(\sigma)) &\subseteq g' C \Lambda C^{-1} \Lambda\cdot \E D^\circ\\
    &\subseteq g' C \Lambda C'^{-1} \cdot \E D^\circ \\
    &\subseteq g' \Lambda C'C'^{-1}\cdot \E D^\circ \\
    &\subseteq g'\Lambda\cdot \E(C'C'^{-1}D^\circ).
\end{align*}

	If $\sigma$~is a nondegenerate $n$-simplex of~$A$, we let $g \in \Gamma_\chi$ be any element for which $\eta(\sigma) \subseteq g\cdot \E C$, so in particular $g$~satisfies Condition~1. Using the fact that $\tilde \mu(\SD^m(\sigma)) = \eta(\sigma)$ and that $C\subseteq D^\circ$, we see that Condition~2 also holds for~$g$. Thus, the same argument as before establishes~$(\star)$.

    Finally, we again use that $\Lambda$~commutes with compacts to find $D\in\C(\Gamma)_{\supseteq D^\circ}$ such that $\Lambda D^\circ \subseteq D\Lambda$.
    
    We are left to show that for each commutative square $(X,A, \eta, \mu)$, the maps $\tilde\mu$ constructed above are such that the (unique) simplicial homotopy
	$J\colon \SD^m(X) \times \Delta^1 \to \E (\Gamma / \Lambda)$ from~$\mu\circ \Phi^m$ to~$p\circ \tilde\mu$ has image in $\Gamma_\chi \cdot \E p(D)$.
    To that end, we observe that for each simplex~$\sigma$ of~$X$, by our construction there is $g\in \Gamma_\chi$ such that
    \begin{enumerate}
    \item $\mu\circ \Phi^m(\SD^m(\sigma))\subseteq g\cdot \E p(C) \subseteq g\cdot \E p(D)$, and
    \item $\tilde\mu(\SD^m(\sigma))\subseteq g \Lambda \cdot \E D^\circ \subseteq g\cdot \E (\Lambda D^\circ) \subseteq g\cdot \E (D\Lambda) $.
    \end{enumerate}
    From the second condition, we deduce that $$p\circ \tilde \mu(\SD^m(\sigma)) \subseteq g\cdot \E p(D\Lambda) = g\cdot \E p(D).$$
	In other words, both maps take $\SD^m(\sigma)$ into the free simplicial set $g\cdot \E p(D)$, which thus contains all simplices of $J(\SD^m(\sigma) \times \Delta^1)$. The conclusion then follows.
\end{proof}

As a consequence, we obtain a characterization of low-dimensional homotopical $\Sigma$-sets of~$\Gamma$ in terms of the filtration $(\Gamma_\chi \cdot \E F)_{F\in \C(\Gamma / \Lambda)}$, generalizing an earlier theorem \cite[Theorem~8.2]{BHQ24a}.

\begin{prop}[$\TopS^k(\Gamma)$~via $\Gamma / \Lambda$]\label{prop:sigmaofquotient} If $\Lambda$~is of type~$\mathrm C_n$, then:

    \begin{enumerate}
        \item For every $k\le n$, if the filtration $(\Gamma_\chi \cdot \E F)_{F\in\C(\Gamma /\Lambda)}$ is essentially $(k-1)$-connected, then $\chi \in \TopS^k(\Gamma)$.

        \item For every $k\le n+1$, if $\chi\in \TopS^{k}(\Gamma)$, then  $(\Gamma_\chi \cdot \E F)_{F\in\C(\Gamma/\Lambda)}$ is essentially $(k-1)$-connected.
    \end{enumerate}
\end{prop}

\begin{proof}
(1) To show that $(\Gamma_\chi \cdot \E C)_{C\in\C(\Gamma)}$ is essentially $(k-1)$-connected, let $C \in \C(\Gamma)$. The hypothesis, together with Lemma~\ref{lem:reindexing},  yields $C'\in \C(\Gamma)_{\supseteq C}$ such that the inclusion $\Gamma_\chi \cdot \E p(C) \into \Gamma_\chi \cdot \E p(C')$ is $\pi_i$-trivial for all $i\le k-1$. Choose $D\in \C(\Gamma)_{\supseteq C'}$ satisfying the lifting property in Lemma~\ref{lem:lift} with respect to~$C'$. We claim that the inclusion $\Gamma_\chi \cdot \E C \into \Gamma_\chi \cdot \E D$ is $\pi_i$-trivial for all $i\le k-1$. To show this, we will make use of the combinatorial description of the homotopy groups~$\pi_i$ given by Lemma~\ref{lem:comb_pi_k}.
    
    Start with a simplicial map $\eta \colon \SD^m(\partial \Delta^{i+1}) \to \Gamma_\chi \cdot \E C$ from a subdivision of $\partial\Delta^{i+1}$. By choice of~$C'$, there are $m'\ge m$ and a map $\mu\colon \SD^{m'}(\Delta^{i+1}) \to \Gamma_\chi \cdot \E p(C')$ filling~$p \circ \eta$, that is, $\mu$~extends the composition 
    \[\SD^{m'}(\partial \Delta^{i+1}) \xrightarrow{\Phi^{m'-m}}\SD^m(\partial \Delta^{i+1}) \xrightarrow{\eta} \Gamma_\chi \cdot \E C \xrightarrow{p} \Gamma_\chi \cdot \E p(C).\]
    Since $\Delta^{i+1}$ has dimension $i+1\le k\le n$, we may apply the defining property of~$D$ to find a lift $\tilde \mu \colon \SD^{m''}(\Delta^{i+1}) \to \Gamma_\chi \cdot \E D$ of~$\mu$ extending $\eta \circ \Phi^{m''-m}$. This exhibits~$\eta$ as trivial in $\pi_i(\Gamma_\chi \cdot \E D)$.

    (2) Consider a compact subset of $\Gamma / \Lambda$, which is of the form~$p(C)$ for some $C\in\C (\Gamma)$ by Lemma~\ref{lem:reindexing}. Choose $D\in\C(\Gamma)_{\supseteq C}$ as given by Lemma~\ref{lem:lift}, and then use the hypothesis that $\chi \in \TopS^k(\Gamma)$ to find $D'\in \C(\Gamma)_{\supseteq D}$ such that the inclusion $\Gamma_\chi \cdot \E D \into \Gamma_\chi \cdot \E D'$ is $\pi_i$-trivial for all $i\le k-1$. We will show that the inclusion $\Gamma_\chi \cdot \E p(C) \into \Gamma_\chi \cdot \E p(D')$ is $\pi_i$-trivial for all $i\le k-1$.

    Given $\eta \colon \SD^m(\partial \Delta^{i+1}) \to \Gamma_\chi \cdot \E p(C)$, since $\partial \Delta^{i+1}$ has dimension $i\le k-1\le n$, we may apply the lifting property of~$D$ to the square
    \[\begin{tikzcd}
        \emptyset \ar[r] \ar[d,hook]& \Gamma_\chi \cdot \E C \ar[d,"p"]\\
        \SD^m(\partial \Delta^{i+1}) \ar[r,"\eta"]& \Gamma_\chi \cdot \E p(C)
    \end{tikzcd}.\]
    This yields a lift $\tilde \eta \colon \SD^{m'}(\partial \Delta^{i+1}) \to\Gamma_\chi \cdot \E D$. By definition of~$D'$, one may then find a filling $\mu \colon \SD^{m''}(\partial \Delta^i) \to \Gamma_\chi \cdot \E D'$, whose postcomposition with~$p$ shows that $\eta$~is nullhomotopic in $\Gamma_\chi \cdot \E p(D')$.
\end{proof}

\subsubsection{Homological lifting}\label{sec:homological_lifting}

The the material in Section~\ref{sec:homological_lifting} is adapted from and generalizes a theorem in the article on homological $\Sigma$-sets for locally compact groups \cite[Theorem~10.4]{BHQ24b}. Throughout, all chain complexes and homology modules will be taken over the coefficient ring~$R$, but for simplicity, we will suppress~$R$ from the notation.

Before tackling the main results of this section, we introduce a piece of notation. Given a vertex~$x$ and a $k$-simplex~$\sigma=(x_0, \ldots, x_k)$ of a free simplicial set~$\E X$, we define the $(k+1)$-simplex $$[x,\sigma]:=(x,x_0, \ldots, x_k).$$
This notation is extended $R$-linearly to all chains in the simplicial chain complex~$\Ch(\E X)$. Intuitively, given a $k$-chain~$c$, the $(k+1)$-chain $[x,c]$~may be visualized as a cone with base~$c$ and vertex~$x$.

\begin{lem}[Boundary of a cone]\label{lem:vertex_prepend}
Let $X$~be a set and $x\in X$. Then for every $k\in\NN$ and every $k$-chain $c\in \Ch_k(\E X)$, we have
    $$\partial [x,c] = c-[x,\partial c].$$
\end{lem}
\begin{proof}
    It suffices to prove this for $c=\sigma$ a simplex, in which case the statement is immediate from the definition of~$\partial$.
\end{proof}

\begin{lem}[Lifting along $\Lambda$~of type~$\mathrm {CP}_n(R)$]\label{lem:homological_lift}
    If $\Lambda$~is of type~$\mathrm{CP}_n(R)$, then for every $C\in \C(\Gamma)$, there exist $D\in \C(\Gamma)_{\supseteq C}$ and a map of $R$-chain complexes
    \[\varphi \colon \Ch(\Gamma_\chi \cdot \E p(C))^{(n)} \to \Ch(\Gamma_\chi\cdot \E D)\]
    extending $\mathrm{id} \colon R \to R$ and making the following diagram commute up to homotopy of $R$-chain complexes:
    \[\begin{tikzcd}
        \Ch(\Gamma_\chi \cdot \E C)^{(n-1)} \ar[r, hook] \ar[d, "p"]& \Ch(\Gamma_\chi\cdot \E D)\ar[d, "p"]\\
        \Ch(\Gamma_\chi \cdot \E p(C))^{(n)} \ar[r, hook] \ar[ur,"\varphi"]& \Ch(\Gamma_\chi\cdot \E p(D))
    \end{tikzcd}.\]
\end{lem}

\begin{proof}
    Since $\Lambda$~almost commutes with compacts, choose $C'\in \C(\Gamma)_{\supseteq C}$ satisfying $C\Lambda\subseteq \Lambda C'$ and $C^{-1}\Lambda \subseteq \Lambda C'^{-1}$.
    We will use induction on~$n$ to construct the set~$D$, the map~$\varphi$, and homotopies~$h, j$ witnessing, respectively, commutativity of the left and right triangle. Concurrently, we will establish the following auxiliary properties:

    \begin{itemize}
		\item[($\star$)]        
        For every $g'\in \Gamma_\chi$, we have:
        \begin{enumerate}
            \item $\varphi_n$ maps $\Ch_n(g'\cdot\E p(C))$ into $\Ch_n(g'\Lambda \cdot \E D)$,
            \item $h_{n-1}$ maps $\Ch_{n-1}(g'\cdot \E C)$ into $\Ch_n(g'\Lambda \cdot \E D)$, and
        \item $j_n$ maps $\Ch_n(g'\cdot\E p(C))$ into $\Ch_{n+1}(g' \cdot\E p(D))$.
        \end{enumerate}
	\end{itemize}
    
    We start the induction at $n=-1$, setting $D:=C$, taking $\varphi_{-1} \colon R \to R$ to be the identity (as we must), and $j_{-1}$ to be the zero map. All other maps of $\varphi$, $h$, and $j$ are necessarily zero. Obviously $\varphi$~is a chain map, $h$~and~$j$ witness commutativity of the respective triangles (the former being a vacuous statement), and $(\star)$~is satisfied.

    Suppose now that $n\ge 0$, and that we have already defined $D_-\in\C(\Gamma)_{\supseteq C}$ and constructed a chain map $\varphi\colon \Ch(\Gamma_\chi \cdot \E p(C))^{(n-1)} \to \Ch(\Gamma_\chi\cdot \E D_-)$ and chain homotopies~$h,j$ witnessing commutativity of the triangles, with $(\star)$~being satisfied.

    Considering the (proper) left action of~$\Lambda$ on~$\Gamma$, we use the $\chi=0$ case of Proposition~\ref{prop:proper_actions} to characterize property~$\mathrm {CP}_n(R)$ on~$\Lambda$ as essential triviality of the directed system of reduced homology groups $(\redH_{k}(\Ch (\Lambda\cdot \E K)))_{K\in \C(\Gamma)}$ for all $k\le n-1$. So let $D^\circ \in \C(\Gamma)_{\supseteq D_-}$ be large enough that the inclusion-induced map 
    \[ \redH_{n-1}(\Ch (\Lambda \cdot \E D_-)) \to \redH_{n-1}(\Ch(\Lambda\cdot \E D^\circ))\]
    is trivial. Using  the assumption that $\Lambda$~almost commutes with compacts, we define the desired set $D\in \C(\Gamma)_{\supseteq C'C'^{-1}D^\circ}$ to be such that $\Lambda C'C'^{–1}D^\circ\subseteq D\Lambda$.
    
    We extend~$\varphi$ to the $n$-skeleton~$\Ch(\Gamma_\chi \cdot \E p(C))^{(n)}$ by defining a new $R$-module map
    $$\varphi_n \colon \Ch_n(\Gamma_\chi \cdot \E p(C))\to \Ch_n(\Gamma_\chi\cdot \E D^\circ),$$
    and then postcomposing all maps $\varphi_k$ with the inclusions into $\Ch_k(\Gamma_\chi \cdot \E D)$.
    To define $\varphi_n$ on an $n$-simplex~$\sigma$ of $\Gamma_\chi \cdot \E p(C)$,  choose $g\in \Gamma_\chi$ such that $\sigma$~lies in $g\cdot \E p(C)$. Since $\partial \sigma \in \Ch_{n-1}(g\cdot\E p(C))$, we see from $(\star)$ that
    \[\varphi_{n-1}(\partial\sigma)\in \Ch_{n-1}(g \Lambda\cdot \E D_-).\]
    Now, $\varphi_{n-1}(\partial \sigma)$ is a cycle of this $g$-translated chain complex (for $\partial(\varphi_{n-1}(\partial \sigma)) = \varphi_{n-1}(\partial^2\sigma) =0$), so by our choice of~$D^\circ$, there is an $n$-chain $c\in \Ch_n(g\Lambda\cdot \E D^\circ)$ with $\partial c = \varphi_{n-1}(\partial \sigma)$. We put $\varphi_n(\sigma) := c$, and as the $n$-simplices form a free $R$-basis, this defines an $R$-module map~$\varphi_n$. It is straightforward from the construction that $\varphi_n$ extends $\varphi$ to a chain map.

    To prove the first claim of~$(\star)$, let $g'\in \Gamma_\chi$ be such that $\sigma \in g'\cdot \E p(C)$ (possibly not the same as $g$~from the previous paragraph). Writing the $0$-th vertex of~$\sigma$ as $g\cdot p(x) =g'\cdot p(x')$ with $x,x'\in C$, we see that for some $l\in \Lambda$, we have $g = g'x'lx^{-1}$, whence
    \begin{align*}
        \varphi_n(\sigma) &\in \Ch_n(g'x'lx^{-1}\Lambda\cdot \E D^\circ)\\
        &\subseteq \Ch_n(g'\Lambda C'C'^{-1}\cdot \E D^\circ)\\
        &\subseteq \Ch_n(g'\Lambda \cdot \E (C'C'^{-1}D^\circ))\\
        &\subseteq \Ch_n(g'\Lambda \cdot \E D).
    \end{align*}

    Next, we redefine the zero map~$h_{n-1}$ of the given~$h$. For each $(n-1)$-simplex~$\tau$ of $\Gamma_\chi \cdot\E C$, choose $g\in \Gamma_\chi$ such that $\tau \in g\cdot \E C$. Using $(\star)$, we see by inspecting each summand that the chain
    \[z:= \varphi_{n-1} (p(\tau))-\tau - h_{n-2}(\partial \tau)\] 
    lies in~$\Ch_{n-1}(g\Lambda\cdot \E D_-)$.
    A straightforward computation then shows $\partial z=0$, whence there is $c\in \Ch_n(g\Lambda\cdot \E D^\circ)$ with $\partial c =z$. Putting $h_{n-1}(\tau) := c$, it follows immediately that the new~$h$ is indeed a chain homotopy as required.

    As before, an expression of~$\tau$ using a different $g'\in \Gamma_\chi$ yields $g\in g'CC^{-1}$, and again we obtain
    \begin{align*}
        h_{n-1}(\tau) & \in\Ch_n(g'CC^{-1}\Lambda\cdot \E D^\circ) \\
        &\subseteq \Ch_n(g'\Lambda\cdot \E (C'C'^{-1}D^\circ))\\
        &\subseteq \Ch_n(g'\Lambda\cdot \E D),
    \end{align*}
    establishing the second condition of~$(\star)$.

    Lastly, we explain how to modify the zero map~$j_n$ in the given chain homotopy~$j$. For each $n$-simplex $\sigma$ of $\Gamma_\chi\cdot \E p(C)$, denote its $0$-th vertex by~$x$. We shall define
    \[j_n(\sigma):=[x, p(\varphi(\sigma))-\sigma - j_{n-1}(\partial \sigma)],\]
    a formula which requires some explanation -- first and foremost as to why the expression on the right hand side lies in $\Ch_{n+1}(\Gamma_\chi \cdot \E p(D))$, since a priori, it is only clear that it lies in $\Ch_{n+1}(\E(\Gamma/\Lambda))$. That is however immediate once we establish the third statement in~$(\star)$, so let us do that right away.

    Let $g'\in \Gamma_\chi$ be such that $\sigma$~lies in~$g'\cdot \E p(C)$. Since $g'\cdot \E p(D)$~is a free simplicial set containing~$x$, we need only show that all three summands $p(\varphi(\sigma))$, $\sigma$, and $j_{n-1}(\partial \sigma)$ lie in $\Ch_{n}(g'\cdot \E p(D))$. For $\sigma$, this is clear, and for $j_{n-1}(\partial \sigma)$ it follows from the case $n-1$ of~$(\star)$. Finally, recall that earlier, when verifying the first condition of~$(\star)$, we actually proved that
    \begin{align*}
        \varphi_n(\sigma) &\in \Ch_n(g'\Lambda \cdot \E (C'C'^{–1}D^\circ))\\
        &\subseteq \Ch_n(g' \cdot \E (\Lambda C'C'^{–1}D^\circ))\\
        &\subseteq \Ch_n(g' \cdot \E (D\Lambda)).
    \end{align*}
    Therefore, $p(\varphi_n(\sigma)) \in \Ch_n(g'\cdot \E p(D))$, as claimed.

    All that is left to do is show that the new~$j$ is a chain homotopy from the inclusion to $p\circ \varphi$.
    And indeed, by expanding the first summand using Lemma~\ref{lem:vertex_prepend} and doing some straightforward manipulations, we see
    \begin{align*}
    \partial j_n(\sigma) + j_{n-1}(\partial \sigma) &= p(\varphi(\sigma))- \sigma - [x, \partial p  (\varphi(\sigma)) - \partial\sigma - \partial j_{n-1} (\partial \sigma)]\\
    &= p(\varphi(\sigma))- \sigma - [x,  p  (\varphi(\partial\sigma)) - \partial\sigma - \partial j_{n-1} (\partial \sigma)].
    \end{align*}
    The chain homotopy formula one dimension below, $\partial j_{n-1} + j_{n-2}\partial = p \circ \varphi_{n-1} - \operatorname{id}$, which we are inductively assuming, shows that the expression in brackets equals $j_{n-2}(\partial^2\sigma) = 0$.
    \end{proof}

We can now give a homological counterpart to Proposition~\ref{prop:sigmaofquotient}.

\begin{prop}[$\TopS^k(\Gamma;R)$ via $\Gamma / \Lambda$]\label{prop:sigmaofquotient_homological}
If $\Lambda$ is of type~$\mathrm{CP}_n(R)$, then:
    \begin{enumerate}
    \item For every $k\le n$, if the filtration $(\Gamma_\chi \cdot \E F)_{F\in\C(\Gamma /\Lambda)}$ is essentially $(k-1)$-acyclic over~$R$, then $\chi \in \TopS^k(\Gamma;R)$.

    \item For every $k\le n+1$, if $\chi\in \TopS^{k}(\Gamma;R)$, then  $(\Gamma_\chi \cdot \E F)_{F\in\C(\Gamma/\Lambda)}$ is essentially $(k-1)$-acyclic over~$R$.
     
    \end{enumerate}
\end{prop}

\begin{proof} 
    (1) To show that $(\Gamma_\chi\cdot \E C)_{C\in \C(\Gamma)}$ is essentially $(k-1)$-acyclic, let $C\in \C(\Gamma)$. Using the hypothesis, together with Lemma~\ref{lem:reindexing}, let $C'\in \C(\Gamma)_{\supseteq C}$~be such that for all $i\le k-1$ the inclusion-induced map
    $\redH_{i}(\Gamma_\chi\cdot \E p(C)) \to \redH_{i}(\Gamma_\chi\cdot \E p(C'))$
    is trivial. Then, apply Lemma~\ref{lem:homological_lift} to produce $D\in \C(\Gamma)_{\supseteq C'}$ and $\varphi\colon \Ch(\Gamma_\chi \cdot\E p(C'))^{(n)} \to \Ch(\Gamma_\chi \cdot \E D)$ with the properties stated therein. We will show that the inclusion-induced map $\redH_i(\Gamma_\chi \cdot \E C)\to \redH_i(\Gamma_\chi \cdot \E D)$ is trivial.

    Consider the diagram of chain complexes
    \[\begin{tikzcd}
        \Ch(\Gamma_\chi \cdot \E C)^{(n-1)} \ar[r,hook] \ar[d,"p"] & \Ch(\Gamma_\chi \cdot \E C')^{(n-1)} \ar[r,hook] \ar[d,"p"] & \Ch(\Gamma_\chi \cdot \E D) \\
        \Ch(\Gamma_\chi \cdot \E p(C))^{(n)} \ar[r,hook] & \Ch(\Gamma_\chi \cdot \E p(C'))^{(n)} \ar[ru,"\varphi"']
    \end{tikzcd},\]
    where the square is commutative and the triangle commutes up to chain homotopy.
    For all $i\le k-1$, the lower horizontal map is trivial on $i$-th homology, hence so is the composition of the two upper maps. This composition can however also be factored as
    \[ \Ch(\Gamma_\chi \cdot \E C)^{(n-1)} \hookrightarrow \Ch(\Gamma_\chi \cdot \E C) \hookrightarrow \Ch(\Gamma_\chi \cdot \E D).\]
    Since $i\le k-1 \le n-1$, the first map is surjective on $i$-th homology. Therefore, the second map is trivial on $i$-th homology, as required.

    (2)
    We start with a compact subset of~$\Gamma / \Lambda$, which, by Lemma~\ref{lem:reindexing},
    may be assumed to be of the form~$p(C)$ with $C\in \C(\Gamma)$. Let $D\in \C(\Gamma)_{\supseteq C}$ and $\varphi \colon \Ch (\Gamma_\chi \cdot \E p(C))^{(n)} \to \Ch (\Gamma_\chi \cdot \E D)$ be as in Lemma~\ref{lem:homological_lift}. Then, choose $D'\in \C(\Gamma)_{\supseteq D}$ such that for all $i\le k-1$ the map $\redH_i(\Gamma_\chi \cdot \E D) \to \redH_i(\Gamma_\chi \cdot \E D')$ is trivial. We claim that the map $\redH_i(\Gamma_\chi \cdot \E p(C)) \to \redH_i(\Gamma_\chi \cdot \E p(D))$ is trivial for all $i\le k-1$.
    
    Consider the following diagram, which is commutative up to chain homotopy:
    \[\begin{tikzcd}
        & \Ch(\Gamma_\chi \cdot \E D) \ar[r,hook] \ar[d,"p"] & \Ch(\Gamma_\chi \cdot \E D') \ar[d,"p"] \\
        \Ch(\Gamma_\chi \cdot \E p(C))^{(n)} \ar[ru,"\varphi"] \ar[r,hook] & \Ch(\Gamma_\chi \cdot \E p(D)) \ar[r,hook]&  \Ch(\Gamma_\chi \cdot \E p(D')) 
    \end{tikzcd}.\]
    For all $i\le k-1$, the upper horizontal map is trivial on $i$-the homology, and thus so is the composition of the lower maps. This composition factors as
    \[ \Ch(\Gamma_\chi \cdot \E p(C))^{(n)} \hookrightarrow \Ch(\Gamma_\chi \cdot \E p(C)) \hookrightarrow \Ch(\Gamma_\chi \cdot \E p(D')),\]
    and since $i \le k-1 \le n$, the first map is surjective on $i$-th homology. We thus conclude the second map is trivial on $i$-th homology, as desired.
\end{proof}

\subsection{From $\Gamma / \Lambda$ to $\Gamma \slcpt\Lambda$}\label{sec:part2}

For this section, we introduce the assumption that the closed commensurated subgroup $\Lambda \subseteq \Gamma$ is open, so $(\Gamma, \Lambda)$~is a Hecke pair and we may define the Schlichting completion $\alpha \colon \Gamma \to \Gamma\slcpt\Lambda =:\G$. Recall that $\L:=\overline{\alpha(\Lambda)}$~is a compact open subgroup of~$\G$.

\begin{prop}[Quotients and Schlichting completions]\label{prop:quotient_to_slcpt}
    The filtrations
    \[(\Gamma_\chi\cdot \E C)_{C\in \C(\G)} \quad\text{and} \quad (\Gamma_\chi\cdot \E F)_{F\in \C(\Gamma /\Lambda)}\]
    of  $\E \G$ and  $\E(\Gamma / \Lambda)$, respectively, are simplicially homotopy-equivalent.
\end{prop}
    
\begin{proof}
    Recall the $\Gamma$-equivariant continuous map $\ev\colon \G\to \Gamma/\Lambda$. We claim that its induced map $\varphi\colon\E \G \to E(\Gamma / \Lambda)$ restricts, for each $C\in\C(\G)$,  to a simplicial map
    \[\varphi_C\colon \Gamma_\chi \cdot \E C  \to \Gamma_\chi \cdot \E (\ev(C)),\]
    and these~$\varphi_C$ assemble to a homotopy equivalence between the filtrations in the statement.

    Note first that since $\ev$~is continuous, each set~$\ev(C)$ is indeed compact. To see that the image of $\Gamma_\chi\cdot \E C$ does indeed lie in $\Gamma_\chi\cdot \E (\ev (C))$, consider a simplex $g\cdot (\xi_i)_{i=0}^k$ of $\Gamma_\chi \cdot \E C$ (with $g\in \Gamma_\chi$ and $\xi_i\in C$). Its image is the simplex $g\cdot (\ev(\xi_i))_{i=0}^k$, which lies in $\Gamma_\chi\cdot\E (\ev(C))$ as required. Note also that since the~$\varphi_C$ were defined as restrictions of a map on all of $\E \G$, they automatically commute with the inclusion maps of the filtrations.

    To construct a homotopy inverse $(\psi_F)_{F\in \C(\Gamma / \Lambda)}$ to~$(\varphi_C)_{C\in \C (\G)}$,  we first choose once and for all a representative of each coset in~$\Gamma/\Lambda$. We shall denote the representative of~$g\Lambda$ by~$[g\Lambda]$, or simply~$[g]$, and given a subset $F\subseteq \Gamma / \Lambda$, we write $[F]:=\{[g]\in \Gamma \mid g\Lambda \in F\}$.
    
    Consider the (generally not continuous or $\Gamma$-equivariant) map of sets 
    \begin{align*}
        \Gamma / \Lambda &\to \G\\
        g\Lambda & \mapsto \alpha([g]),
    \end{align*}
    which induces a simplicial map $\psi \colon \E(\Gamma / \Lambda) \to \E \G$.
    We claim that for each $F\in \C(\Gamma /\Lambda)$, the map~$\psi$ restricts to a map of simplicial sets
    \[\psi_F \colon \Gamma_\chi \cdot \E F \to \Gamma_\chi \cdot \E(\alpha([F]) \L),\]
    and that this family $(\psi_F)_{F\in\C(\Gamma / \Lambda)}$ forms a simplicial homotopy inverse to $(\varphi_C)_{C\in \C(\G)}$.

    First note that since $\Gamma/\Lambda$~is discrete, each $F\in \C(\Gamma / \Lambda)$ is finite, and thus also $\alpha([F])$ is finite, whence $\alpha([F])\L \subseteq \G$ is compact.
    Let us check that each simplex $g\cdot(g_i\Lambda)_{i=0}^k$ of $\Gamma_\chi\cdot \E F$ (with $g\in \Gamma_\chi$ and $g_i\Lambda \in F)$ is indeed mapped to a simplex of $\Gamma_\chi \cdot \E (\alpha([F]) \L)$. 
    We have
    \begin{align*}
        \psi_F(g\cdot(g_i\Lambda)_{i}) & = \psi_F((gg_i\Lambda)_i)\\
        &=(\alpha([gg_i]))_i\\
        &= (\alpha(gg_ih_i))_i & \text{for some $h_i\in \Lambda$}\\
        &=(\alpha(g[g_i]h_i'h_i) )_i& \text{for some $h_i, h_i'\in \Lambda$}\\
        &=g\cdot (\alpha([g_i])\alpha (h_i'h_i) )_i & \text{for some $h_i, h_i'\in \Lambda$}&,
        \end{align*}
which lies in~$\Gamma_\chi \cdot \E (\alpha([F]) \L)$.

 As before, the~$\psi_F$ obviously commute with all inclusion maps.
To see that they assemble to a homotopy inverse to the~$\varphi_C$, note first that the composition $\varphi \circ \psi$~is sends $g\Lambda\mapsto [g] \mapsto \ev([g]) = [g]\cdot \Lambda = g\Lambda$, so it is the identity. Hence the same is true for the restriciton to each $\Gamma_\chi \cdot \E F$.

Now let $C\in \C(\Gamma)$ and consider effect of the other composition on a simplex $g\cdot (\xi_i)_{i=0}^k$. We have:
    $$\begin{array}{ccccc}
        \Gamma_\chi \cdot \E C & \xrightarrow{\varphi_C} & \Gamma_\chi \cdot \E (\ev(C)) &\xrightarrow{\psi_{\ev(C)}}& \Gamma_\chi \cdot \E (\alpha([\ev(C)])\L)\\
        g\cdot(\xi_i)_i & \mapsto&  g\cdot(\xi_i(\Lambda))_i &\mapsto & g\cdot (\alpha([\xi_i(\Lambda)]) \alpha(h_i'h_i))_i
    \end{array}$$
    for some $h_i, h_i' \in \Lambda$.
    Both the given simplex and its image lie in the free simplicial set $g\cdot \E D$, where $D:=C\cup\alpha([\ev(C)])\L$. Thus, the (unique) simplicial homotopy from the inclusion $\iota_C\colon \Gamma_\chi \cdot \E C \into \E \G$ to the composition $\iota_{\alpha([\ev(C)])\L}\circ \psi_{\ev(C)} \circ \varphi_C$ thus has image in $\Gamma_\chi \cdot \E D$. Thus the~$\psi_F$ do indeed form a simplicial homotopy inverse to the $\varphi_C$. 
\end{proof}

Recalling from Corollary~\ref{cor:char_space} that precomposition with~$\alpha$ induces an isomorphism between the character space of~$\G$ and the space of characters in~$\Gamma$ that vanish on~$\Lambda$, we denote by~$\tilde\chi\colon \G \to \RR$ the character with $\chi = \tilde\chi \circ \alpha$.

\begin{lem}[Completing the orbits]\label{lem:complete_orbits}
    The filtrations
    \[(\Gamma_\chi \cdot \E C)_{C\in \C(\G)} \quad \text{and} \quad (G_{\tilde\chi} \cdot \E C)_{C\in \C(\G)}\]
    of $\E \G$ are cofinal.
\end{lem}

\begin{proof}
    For each $C\in \C(\G)$, we obviously have
    $$\Gamma_{\chi} \cdot \E C = \alpha(\Gamma_\chi) \cdot\E C \subseteq \G_{\tilde\chi} \cdot \E C.$$

    For the converse direction, note first that since the cosets of~$\L$ form an partition of~$\G$ by open subsets and $\alpha(\Gamma)$~is dense in~$\G$, we have $\G=\alpha(\Gamma)\L$. What is more, since $\L$~is compact, $\tilde\chi$~is constant on these cosets, and so we have $\G_{\tilde\chi} = \alpha(\Gamma_\chi)\L$. Therefore, for each $C\in\C(\G)$, we see that
    \[\G_{\tilde\chi} \cdot \E C = \alpha(\Gamma_\chi )\L\cdot \E C \subseteq \alpha(\Gamma_\chi)\cdot \E (\L C) = \Gamma_\chi \cdot \E(\L C),\]
    where $\L C\in \C(\G)$.
\end{proof}

Putting everything together, we obtain:

\begin{thm}[$\Sigma$-sets of TDLC completions]\label{thm:main}
    Let $\phi\colon \Gamma \to G$~be a TDLC completion of the Hecke pair~$(\Gamma, \Lambda)$, and $\bar\chi\colon G\to \RR$ the character induced by~$\chi$.
    
    If $\Lambda$ is of type~$\mathrm{C}_n$, then:
    \begin{enumerate}
        \item for every $k\le n$, if $\bar\chi \in \TopS^k(G)$, then $\chi \in \TopS^k(\Gamma)$,
        \item for every $k\le n+1$, if $\chi \in \TopS^k(\Gamma)$, then $\bar\chi \in \TopS^k(G)$.
    \end{enumerate}

    Similarly, if $\Lambda$~is of type~$\mathrm{CP}_n(R)$, then:
     \begin{enumerate}
        \item for every $k\le n$, if $\bar\chi \in \TopS^k(G;R)$, then $\chi \in \TopS^k(\Gamma;R)$,
        \item for every $k\le n+1$, if $\chi \in \TopS^k(\Gamma;R)$, then $\bar\chi \in \TopS^k(G;R)$.
    \end{enumerate}

\end{thm}

\begin{proof}
    By Proposition~\ref{prop:completions_have_same_Sigma}, it suffices to prove the statement in the case where $G = \G$ is the Schlichting completion, so $\bar\chi = \tilde \chi$. We have a sequence of simplicial homotopy equivalences of filtrations
    \begin{align*}
    (\G_{\tilde\chi} \cdot \E C)_{C\in \C(\G)} & \simeq (\Gamma_\chi\cdot \E C)_{C\in \C(\G)} & \text{(Lemma~\ref{lem:complete_orbits})}&\\
    & \simeq (\Gamma_\chi \cdot \E F)_{F\in \C(\Lambda/\Gamma)} & \text{(Proposition~\ref{prop:quotient_to_slcpt})}&,
    \end{align*}
    so all these filtrations have the same essential connectedness / acyclicity properties (by Lemma\ref{lem:homotopyequiv}).
    
    On the one hand, the first of the above filtrations determines which $\Sigma$-sets of~$\G$ contain~$\tilde\chi$.
    On the other hand, since $\Lambda$~almost commutes with compacts (by Proposition~\ref{prop:comm_exchange}), Propositions \ref{prop:sigmaofquotient}~and~\ref{prop:sigmaofquotient_homological} tell us that membership of~$\chi$ in the indicated $\Sigma$-sets may be probed on the latter filtration.
\end{proof}

\section{Applications}\label{sec:applications}
\subsection{Baumslag--Solitar groups}\label{sec:BS}

For each $m,n\in \ZZ \setminus\{0\}$, we consider the Baumslag--Solitar group
$$\Gamma:=\operatorname{BS}(m,n) = \langle a,t \mid ta^mt^{-1} = a^n\rangle,$$
which has $\Lambda:=\langle a\rangle \cong \ZZ$ as a commensurated subgroup. Elder and Willis have studied the Schlichting completion $G_{m,n} := \Gamma \slcpt \Lambda$ \cite{EW18}. In Section~\ref{sec:BS} we will compute their $\Sigma$-sets.

By Corollary~\ref{cor:char_space}, the character space of~$G_{m,n}$ is given by
\[\Homtop(G_{m,n}, \RR) \cong \Hom(\Gamma, \RR)_\Lambda = \RR \tau,\]
where $\tau\colon \Gamma \to \RR$~is the character sending $a \mapsto 0, t  \mapsto 1$. The following result is well-known, but we include a proof for the reader's convenience.
\begin{lem}[The line $\RR \tau$ and $\Sigma$-sets]\label{lem:bs}
We have:
\begin{enumerate}
    \item if $|mn|=1$ then
    $\RR\tau\subseteq \TopS^k(\Gamma)$ for all  $k\ge 1$,
       
    \item if $|m|=1$ and $|n| \ge 2$ then 
    $\RR\tau\cap \TopS^k(\Gamma)=\{\lambda\tau\mid\lambda\le0\}$ for all $k\ge 1$,
     \item if $|m| \geq 2$ and $|n|=1$ then
     $\RR\tau\cap \TopS^k(\Gamma)=\{\lambda\tau\mid\lambda\ge0\}$  for all $k\ge 1$,
      \item if $|m|,|n|\ge 2$ then
    $\RR\tau\cap \TopS^k(\Gamma) = \{0\}$ for all $k\ge 1$.
    \end{enumerate}

    Analogous statements hold for the homological $\Sigma$-sets $\Sigma^k(\Gamma;R)$ over any commutative ring~$R$.
\end{lem}
\begin{proof}
    We start by recalling the well-known description of  the Cayley graph~$Y$ with respect to $\{a,t\}$. The relator $ta^mt^{-1}a^{-n}$ produces a ``basic brick" in the Cayley graph, and the basic bricks are glued together to form an infinite sheet (see Figure~\ref{fig:sheet}).
   \begin{figure}[h] 
\centering
    \def\svgwidth{\linewidth}
    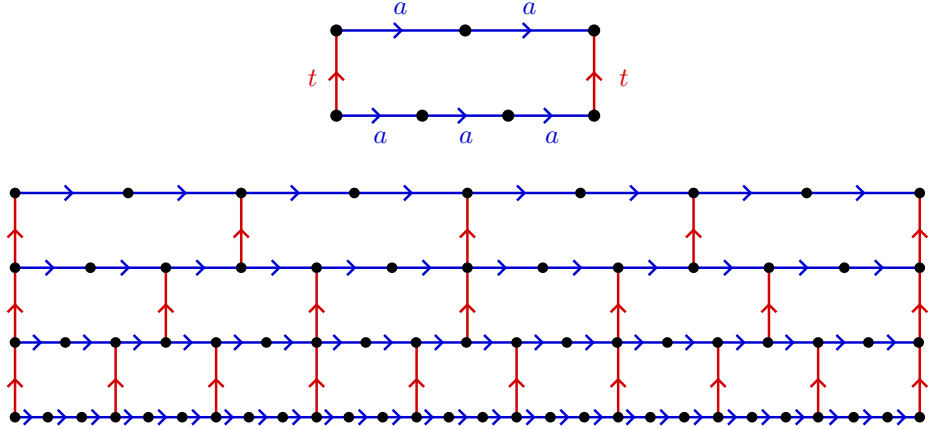

  \caption{A basic brick for constructing the Cayley graph of~$\BS(2,3)$ (top), and a portion of the corresponding infinite sheet (bottom).}
  \label{fig:sheet}
  \end{figure}
    
Finally, the Cayley graph of $\Gamma$ is obtained by gluing $n$~upper half sheets together to $m$~lower half sheets along each infinite path whose edges are labeled only by~$a$ (see Figure~\ref{fig:cg}).
\begin{figure}[h]
    \centering
    \def\svgwidth{0.9 \linewidth}
    %% Creator: Inkscape 1.4.2 (ebf0e940d0, 2025-05-08), www.inkscape.org
%% Accompanies image file 'BS.pdf' (pdf, eps, ps)
\begingroup%
  \makeatletter%
  \providecommand\color[2][]{%
    \errmessage{(Inkscape) Color is used for the text in Inkscape, but the package 'color.sty' is not loaded}%
    \renewcommand\color[2][]{}%
  }%
  \providecommand\transparent[1]{%
    \errmessage{(Inkscape) Transparency is used (non-zero) for the text in Inkscape, but the package 'transparent.sty' is not loaded}%
    \renewcommand\transparent[1]{}%
  }%
  \providecommand\rotatebox[2]{#2}%
  \newcommand*\fsize{\dimexpr\f@size pt\relax}%
  \newcommand*\lineheight[1]{\fontsize{\fsize}{#1\fsize}\selectfont}%
  \ifx\svgwidth\undefined%
    \setlength{\unitlength}{446.0903378bp}%
    \ifx\svgscale\undefined%
      \relax%
    \else%
      \setlength{\unitlength}{\unitlength * \real{\svgscale}}%
    \fi%
  \else%
    \setlength{\unitlength}{\svgwidth}%
  \fi%
  \global\let\svgwidth\undefined%
  \global\let\svgscale\undefined%
  \makeatother%
  \begin{picture}(1,0.54819091)%
    \lineheight{1}%
    \setlength\tabcolsep{0pt}%
    \put(0,0){\includegraphics[width=\unitlength,page=1]{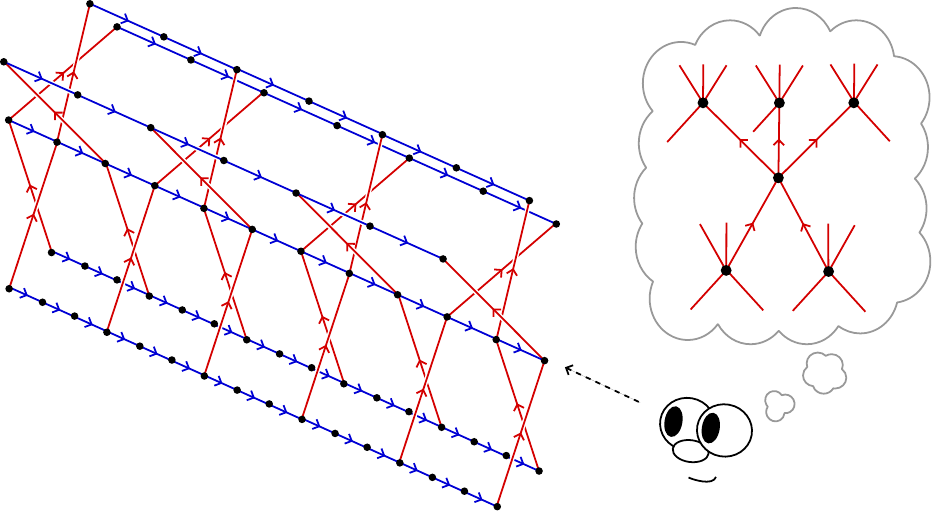}}%
    \put(0.2940384,0.26904727){\color[rgb]{0,0,0.83137255}\makebox(0,0)[t]{\lineheight{0.69999999}\smash{\begin{tabular}[t]{c}$a$\end{tabular}}}}%
    \put(0.26405006,0.32710291){\color[rgb]{0.83137255,0,0}\makebox(0,0)[t]{\lineheight{0.69999999}\smash{\begin{tabular}[t]{c}$t$\end{tabular}}}}%
    \put(0.72673164,0.33257839){\color[rgb]{0,0,0}\makebox(0,0)[t]{\lineheight{0.69999999}\smash{\begin{tabular}[t]{c}$T$\end{tabular}}}}%
  \end{picture}%
\endgroup%

    \caption{The Cayley graph of $\BS(2,3)$ with respect to the generating set~$\{a,t\}$ (left). Filling in the ``bricks'' with $2$-cells, produces a Cayley complex homeomorphic to~$T\times\RR$, where $T$~is a $5$-regular tree (right).}
    \label{fig:cg}
\end{figure}
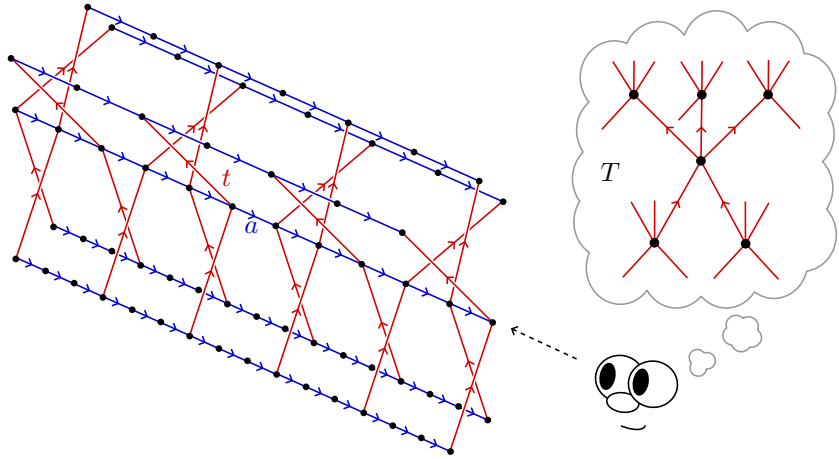
    Completing~$Y$ to a presentation complex, which amounts to attaching a $2$-cell at each brick of the Cayley graph, we obtain a free $\Gamma$-CW-complex~$X$ of finite type, which is homeomorphic to $T\times \RR$, where $T$~is an $(m+n)$-regular tree. In particular, $X$~is contractible, which shows that $\Gamma$~is of type $F_\infty$; in other words, $0\in \Sigma^\infty(\Gamma)$.
    
    Since $\Sigma$-sets are closed under taking positive scalar multiples, we need only determine the membership status of the characters $\pm \tau$ in each $\Sigma^k(\Gamma)$. Note that for each vertex~$v$ of~$T$, the vertices of~$Y$ in the line $\{v\}\times \RR$ have constant $\tau$-value. Moreover, at each vertex of~$T$, there are precisely $n$~edges pointing in the increasing direction of~$\tau$, which we visualise as ``upward'', and $m$~pointing ``downward''. It is then clear that the following two conditions are equivalent:
\begin{enumerate}
    \item[(1)]  $|n|=1$ (resp. $|m|=1$),
    \item[(2)] the subgraph $Y_{\tau \ge 0}$ (resp. $Y_{\tau\le 0}$) is connected.
\end{enumerate}
    Since each connected component of $Y_{\tau \ge 0}$ is contractible, these conditions are also equivalent to
\begin{enumerate}
    \item[(3)] the subcomplex $X_{\tau \ge 0}$ (resp. $X_{\tau\le 0}$) is contractible.
\end{enumerate}
It is easy to check that if some filtration stage is trivial then the filtration is essentially trivial, so  condition (2) implies:
\begin{enumerate}
    \item[(4)] $\tau \in  \Sigma^1(\Gamma)$ (resp. $-\tau \in \Sigma^1(\Gamma)$).
\end{enumerate}
But it is well-known that for the special case of the $\Sigma$-set $\Sigma^1$, the converse implication $(4) \Rightarrow (2)$ holds as well \cite[Satz~2.4]{Ren88}.

We consider one final condition
\begin{enumerate}
    \item[(5)] $\tau \in \Sigma^\infty(\Gamma)$ (resp. $-\tau \in \Sigma^\infty(\Gamma)$).
\end{enumerate}
We obviously have implications $(3)\Rightarrow(5)\Rightarrow(4)$, so all five conditions are equivalent.
 
The assertion on homotopical $\Sigma$-sets is now merely a restatement of the equivalence $(1)\Leftrightarrow (5)$.

To prove the statement for the homological $\Sigma$-sets, we use the fact that $\Sigma^1(\Gamma) = \Sigma^1(\Gamma;R)$, and thus it suffices to show that $\Sigma^1(\Gamma)\subseteq \Sigma^\infty(\Gamma;R)$. But this follows from the implications $(4) \Rightarrow (3) \Rightarrow \chi \in\Sigma^\infty(\Gamma;R)$.
\end{proof}
 
Let $\bar\tau\colon G_{m,n}\to \RR$ be the character induced by~$\tau$ as in Lemma~\ref{lem:mapfromslcpt}.

\begin{prop}[The $\Sigma$-sets of~$G_{m,n}$]\label{prop:BS_main}
  The homotopical $\Sigma$-sets of $G_{m,n}$ are given by:
\begin{enumerate}
    \item if $|mn|=1$ then
    $\TopS^k(G_{m,n})=\RR\bar\tau$ for all  $k\ge 1$,
       
    \item if $|m|=1$ and $|n| \ge 2$ then
    $\TopS^k(G_{m,n}) = \{\lambda\bar\tau \mid \lambda \le 0\}$ for all $k\ge 1$,
     \item if $|m| \geq 2$ and $|n|=1$ then
    $\TopS^k(G_{m,n}) = \{\lambda\bar\tau \mid \lambda \ge 0\}$ for all $k\ge 1$,
      \item if $|m| \geq 2$ and $|n|\ge 2$ then
    $ \TopS^k(G_{m,n}) = \{0\}$ for all $k\ge 1 $.
    \end{enumerate}
    
    Moreover, for every commutative ring $R$, every $m,n\in \ZZ \setminus\{0\}$, and every $k\in \NN$, we have
    \[\TopS^k(G_{m,n};R) = \TopS^k(G_{m,n}).\]
\end{prop}
\begin{proof}
By Corollary~\ref{cor:char_space}, the canonical map $\alpha \colon \Gamma\to G_{m,n}$ induces an isomorphism of $\RR$-vector spaces
$$\alpha^*\colon \Homtop(G_{m,n},\RR) \to \Hom(\Gamma, \RR)_\Lambda,$$
and the right-hand side is easily seen to be $1$-dimensional and spanned by the character~$\tau$. Therefore, $\Homtop(G_{m,n},\RR)$ is spanned by~$\bar\tau$.

Since $\Lambda \cong \ZZ$ is of type $\mathrm{F}_\infty$ (hence also of type $\mathrm {FP}_\infty (R)$), Theorem~\ref{thm:main} states that each $\Sigma$-set $\TopS^k(G_{m,n})$ is precisely the preimage of $\Sigma^k(\Gamma)$ under~$\alpha^*$ (and similarly for the homological $\Sigma$-sets). 
 Hence, Lemma~\ref{lem:bs} directly yields the claim.
\end{proof}
\begin{rem}[Discrete completion]
Raum has shown that the group $G_{m,n}$ is discrete if, and only if, $|m|=|n|$, in which case $G_{m,n}\cong\big(\ZZ/|n|\ZZ\big)\ast\ZZ$ \cite[Theorems 7.2~and~9.2]{R19}.
\end{rem}

\begin{rem}[The case $m=1$]
The TDLC group $G_{1,n}$ with $|n|\ge2$ can also be described using the isomorphism $\Gamma\cong \ZZ[\tfrac 1n]\rtimes \ZZ$, where the semidirect product action is by multiplication with~$n$, and the subgroup~$\Lambda$~corresponds to $\ZZ\subseteq \ZZ[\tfrac 1n]$.

Consider the dense inclusion
    $$\ZZ[\tfrac1n]\rtimes \ZZ\into\biggl(\bigoplus_{p\mid n}\QQ_p\biggr)\rtimes \ZZ=:G.$$
    The TDLC group~$G$ has a compact open subgroup $L:=\bigoplus_{p\mid n}\ZZ_p$, whose preimage corresponds to~$\Lambda$, so Proposition~\ref{prop:schlichting_from_compl} tells us that the above map induces an isomorphism $G_{1,n} \cong G/\operatorname{Core}_{G}(L)$.
    And this normal core is trivial: if $n$~has the prime decomposition $n=\prod_{p\mid n} p^{v_p}$ and we denote by~$t$ the generator of the $\ZZ$-factor, then for every $k\in \ZZ$ we have $t^kLt^{-k} = \bigoplus_{p\mid n}p^{k v_p}\ZZ_p$. Since $v_p \ge 1$ for every~$p$, the intersection of all these conjugates is trivial. We thus conclude that $G_{1,n}\cong G$.
\end{rem}

\subsection{Building on examples of Schesler}\label{sec:Schesler}

Throughout this section, fix $n\in\NN$ and a nonempty finite set~$\P$ of primes. Given a commutative ring~$R$, denote by~$\mathcal B_n(R)$ the subgroup of upper triangular matrices in~$\operatorname{SL}_n(R)$. We will write $\ZZ[\mathcal P^{-1}]$ to denote the subring of~$\QQ$ generated by $\{\tfrac 1 p \mid p\in \mathcal P\}$ and equipped with the discrete topology.

We shall be interested in the discrete groups
$$\Gamma:= \B_n(\ZZ[\P^{-1}]) \qquad \text{and}\qquad \Lambda:=\mathcal B_n(\ZZ).$$
Eduard Schesler has given a partial description of the homotopical $\Sigma$-sets of~$\Gamma$. In this section, we will see that $\Lambda$~is a commensurated subgroup of~$\Gamma$ of type~$\mathrm F_\infty$, compute the Schlichting completion $\G:=\Gamma\slcpt\Lambda$, and use our main result, Theorem~\ref{thm:main}, to transfer Schesler's computations from $\Gamma$ to~$\G$, producing a description of the the sets $\TopS^k(\G)$ that is as complete as Schesler's.

\subsubsection{The Schlichting completion}

Define the topological rings
$$\ZZ_\P:=\prod_{p\in\P}\ZZ_p \qquad \text{and} \qquad \QQ_\P := \prod_{p\in\P} \QQ_p,$$
equipped with the product topology. We regard $\ZZ[\mathcal P^{-1}]$ as a subring of~$\QQ_\P$ via component-wise inclusion.
\begin{lem}[A dense subring]\label{lem:dense_subring}
    The ring $\ZZ[\P^{-1}]$~is dense in~$\QQ_\P$.
\end{lem}
\begin{proof}
    Let $(\alpha_p)_{p\in\P} \in \QQ_\P$, which we will show to be arbitrarily well approximated by elements of~$\ZZ[\P^{-1}]$. Write $v_\alpha:=\min\{\nu_p(\alpha_p) \mid p\in \P\}$ and define the integer
    $$P:=\prod_{\substack{p\in \P\\\nu_p(\alpha_p)<0}} p^{-v_\alpha}.$$
    Note that for every $p\in \P$, we have $\nu_p(P\alpha_p)\ge 0$, that is, $P\alpha \in \ZZ_\P$.

    A well-known application of the Chinese Remainder Theorem shows that $\ZZ$~is dense in~$\ZZ_\P$, and if  $n\in \ZZ$ approximates~$P\alpha$, say with $\nu_p (n-P\alpha) \ge m$ for all $p\in \P$, then
    $$\nu_p (\tfrac n P - \alpha) = \nu_p (\tfrac 1P)+\nu_p (n - P\alpha) \ge v_\alpha + m.$$
    Therefore, if $n$~is chosen close enough to~$P\alpha$, then $\tfrac n P \in\ZZ[\P^{-1}]$ gets arbitrarily close to~$\alpha$.
\end{proof}

We now consider the group $\QQ_\P^\times$ of units in~$\QQ_\P$. Since for every prime~$p$ the operation $\QQ_p ^\times \to \QQ_p^\times$ of taking multiplicative inverses is continuous, also taking inverses in $\QQ_\P^\times$ is continuous, and so $\QQ_\P^\times$~is  a (TDLC) topological group.

\begin{lem}[Units of~{$\ZZ[\P^{-1}]$}] \label{lem:topology_invertibles}
    The group of units $\ZZ[\P^{-1}]^\times$ is a discrete subgroup of~$\QQ_\P^\times$, isomorphic to~$ \ZZ/2 \oplus \ZZ^{|\P|}$.
\end{lem}

\begin{proof}
    The invertible elements of $\ZZ[\P^{-1}]$ are the rationals of the form
    $\pm\prod_{p\in\P}p^{k_p}$, where $k_p\in \ZZ$. This makes it clear that $\ZZ[\P^{-1}]^\times$~has the stated isomorphism type; we are only left to check that it lies in $\QQ_\P$ as a discrete subspace. 

    Fix $x\in\ZZ[\P^{–1}]^\times$. Given any distinct $x' \in \ZZ[\P^{-1}]^\times$, suppose that $\nu_p(x')\neq \nu_p(x)$ for some $p\in\P$. Then,
    $$\nu_p(x-x') = \min\{\nu_p(x), \nu_p(x')\} \le \nu_p(x).$$
    If $\nu_p(x) = \nu_p(x')$ for all $p\in \P$, then $x$~and~$x'$ are each other's negative, and so, for all $p\in \P$, we have $\nu_p(x-x') = \nu_p(2x) \le v_p(x)+1$ (the ``$+1$'' accounting for the possibility where $p=2$).

    In either case, we conclude $\nu_p(x-x')\le \nu_p(x)+1$ for some~$p\in \P$, which shows that, with the topology induced from $\QQ_\P$, the point~$x$ is isolated within~$\ZZ[\P^{-1}]^\times$.
\end{proof}

Now let $G$~be the subgroup of~$\B_n(\QQ_p)$ consisting of the matrices whose diagonal entries lie in~$\ZZ[\P^{-1}]^\times$. In other words, $G$~is the subgroup of
$$\begin{pmatrix}
        \ZZ[\P^{-1}]^\times &\QQ_\P&\cdots&\QQ_\P\\
         &\ddots&\ddots&\vdots\\
        &&\ZZ[\P^{-1}]^\times&\QQ_\P\\
        &&&\ZZ[\P^{-1}]^\times
    \end{pmatrix}$$
comprised of the matrices whose diagonal entries multiply to~$1$. Note that (using Lemma~\ref{lem:topology_invertibles}) the above group of matrices is, topologically, a product of totally disconnected locally compact spaces, and hence a TDLC group.
Since $G$~is a closed subgroup, $G$ too is a TDLC group. 

\begin{prop}[A TDLC completion]
\label{prop:G_TDLC_cplt}
    The inclusion $\Gamma \into G$ is a TDLC completion of the Hecke pair~$(\Gamma, \Lambda)$.
\end{prop}
\begin{proof}
    We first show that $\Gamma$~is a dense subgroup of~$G$.
    Since $\Gamma$ may be described as the subgroup of
    $$\begin{pmatrix}
        \ZZ[\P^{-1}]^\times &\ZZ[\P^{-1}]&\cdots&\ZZ[\P^{-1}]\\
         &\ddots&\ddots&\vdots\\
        &&\ZZ[\P^{-1}]^\times&\ZZ[\P^{-1}]\\
        &&&\ZZ[\P^{-1}]^\times
    \end{pmatrix}$$
    where diagonal entries multiply to~$1$, denseness of~$\Gamma$ follows directly from Lemma~\ref{lem:dense_subring}.

    Let $L\subseteq\B_n(\ZZ_\P)$ be the subgroup whose diagonal entries lie in $\ZZ^\times =\{\pm1\}$, so $L$~consists of the matrices in
    $$\begin{pmatrix}
        \ZZ^\times &\ZZ_\P&\cdots& \ZZ_\P\\
        &\ddots &\ddots& \vdots\\
        &&\ZZ^\times & \ZZ_\P\\
        &&&\ZZ^\times
    \end{pmatrix}$$
    with diagonal entries multiplying to~$1$. Since $\ZZ_\P$ is compact and open in~$\QQ_\P$, we see that $L$~is a compact open (hence commensurated) subgroup of~$G$. Hence, we need only show that $\Lambda = \Gamma \cap L$. But this is a direct consequence of the fact that $\ZZ[\P^{-1}] \cap \ZZ_\P$ consists of the elements of~$\ZZ[\P^{-1}]$ with nonnegative $p$-valuation for all $p\in P$, that is, $\ZZ[\P^{-1}] \cap \ZZ_\P = \ZZ$.
\end{proof}
 
\begin{lem}[Computing the normal core]\label{lem:core_description}
    The normal core~$\operatorname{Core}_G(L)$ is the subgroup
    $$N:=\begin{cases} \{\operatorname{Id\}} & \text{for $n$ odd,}\\
    \{\pm\operatorname{Id}\} &\text{for $n$ even.}
    \end{cases}$$
\end{lem}
(The even/odd distinction is merely a consequence of the fact that for $n$~odd, the matrix $-\mathrm{Id}$ has determinant $-1$ and so does not lie in~$L$.)

\begin{proof}
    $N$~is central in~$G$, and thus normal, so certainly $N\subseteq \operatorname{Core}_G(L)$.

    For the converse inclusion, it suffices we show that for every matrix $A = (a_{i,j}) \in L$ that does not lie in~$N$, there exists a matrix $B = (b_{i,j}) \in G$ with $BAB^{-1} \not\in L$.

    Suppose first that $A$~is not a diagonal matrix, so for some $1\le i < j\le n$ we have $a_{i,j} \neq 0$. Thus, for some $p\in \P$, the $p$-valuation of the $p$-coordinate in~$a_{i,j}$ is an integer~$k\ge1$. We take~$B$ to be the diagonal matrix with $b_{i,i} = p^{-k}$, $b_{j,j} = p^{k}$, and $1$'s elsewhere. One then checks that the $(i,j)$-entry of~$BAB^{-1}$ is $p^{-2k}a_{i,j}$. Its $p$-coordinate has $p$-valuation $-k<0$, so this coordinate is not in~$\ZZ_p$, whence $BAB^{-1} \not \in L$.

    Now consider the case where $A$~is a diagonal matrix. In particular, since $A \not \in N$, we see $A$~has two entries $a_{i,i} = \pm1, a_{j,j}=\mp1$ (with $i<j$) of opposite sign. Choose any $\alpha \in \QQ_\P$ such that $2\alpha \not \in \ZZ_\P$, and let~$B$ be the elementary matrix with entry~$\alpha$ on the $(i,j)$-coordinate. The $(i,j)$-coordinate of~$BAB^{-1}$ is then $\mp2\alpha$, so again $BAB^{-1} \not \in L$.
\end{proof}

We are now equipped to describe the Schlichting completion~$\G=\Lambda \slcpt \Gamma$.
\begin{prop}[The Schlichting completion]
\label{prop:slcpt_description}
    The Schlichting completion of $(\Gamma, \Lambda)$ is given by $\G\cong G / N$, where the structure map $\alpha \colon \Gamma \to \G \cong G/N$ is induced by the inclusion $\Gamma \into G$.
\end{prop}
\begin{proof}
Having established that $\Gamma \into G$ is a TDLC completion (Proposition~\ref{prop:G_TDLC_cplt}) and that $\operatorname{Core}_G(L) =N$ (Lemma~\ref{lem:core_description}), this is a direct application of Proposition~\ref{prop:schlichting_from_compl}.
\end{proof}

\subsubsection{The $\Sigma$-sets of~$\Gamma$}\label{sec:Sigma_of_Gamma}

Before stating Schesler's characterization of the $\Sigma^k(\Gamma)$, let us describe the character space $\Hom(\Gamma, \RR)$.

\begin{lem}[An invisible subgroup]\label{lem:unipotent_vanishing}
    Every character $\chi \colon \Gamma \to \RR$ vanishes on the subgroup   
    $$U=\begin{pmatrix}
        1 & \ZZ[\P^{-1}] &\cdots&\ZZ[\P^{-1}]\\
        &\ddots &\ddots& \vdots\\
        &&1 & \ZZ[\P^{-1}]\\
        &&&1
    \end{pmatrix}.$$
\end{lem}
\begin{proof}
    For all $1\le i <j\le n$ and $\alpha \in \ZZ[\P^{-1}]$, denote by $E_{i,j}^\alpha$ the matrix obtained from $\mathrm{Id}$ by replacing the~$0$ at entry $(i,j)$ with~$\alpha$. Gauss elimination shows that the matrices~$E_{i,j}^\alpha$ generate~$U$, so it suffices to prove that $\chi$~vanishes on all of them.

    So fix $\alpha$ and $i<j$. Choose any prime $p\in \P$ and let $A \in \Gamma$ be the diagonal matrix with $p$ in the $i$-th diagonal entry, $p^{-1}$~in the $j$-th, and $1$'s elsewhere. A simple computation shows that
    $$AE_{i,j}^\alpha A^{-1} = E_{i,j}^{p^2\alpha} = (E_{i,j}^{\alpha})^{p^2}.$$
    Therefore, $\chi(E_{i,j}^\alpha) = p^2 \chi(E_{i,j}^\alpha)$, whence $\chi(E_{i,j}^\alpha)=0$.
\end{proof}

\begin{prop}[The character space of~$\Gamma$] \label{prop:Hom_basis}
    The character space $\Hom(\Gamma, \RR)$ has an $\RR$-basis~$b$ comprised of the $(n-1)|\P|$ characters
    \begin{align*}
    \chi_{k,p} \colon \Gamma  &\to \RR\\
    (a_{i,j})&\mapsto\nu_p(a_{k+1,k+1})-\nu_p(a_{k,k}),
\end{align*}
where $1\le k\le n-1$ and $p\in \P$.
\end{prop}

\begin{proof}
The group~$U$ of Lemma~\ref{lem:unipotent_vanishing} is the kernel of the map
$$\varphi\colon \Gamma \to (\ZZ[\P^{-1}]^\times)^n$$
that reads off the diagonal entries, so $\varphi$~induces an isomorphism
$$\varphi^*\colon \Hom(\operatorname{im(\varphi)},\RR) \xrightarrow{\cong} \Hom(\Gamma,\RR).$$

To understand~$\operatorname{im}(\varphi)$, recall the description of~$\ZZ[\P^{-1}]^\times$ from Lemma~\ref{lem:topology_invertibles}, according to which 
$\ZZ[\P^{-1}]^\times\cong \ZZ/2 \oplus \ZZ^{|\P|}$, with the $\ZZ/2$~factor recording signs, and the $\ZZ$~factors tracking the $p$-valuation for each $p\in \P$. The fact that the diagonals of matrices in~$\Gamma$ are constrained (only) to having product equal to~$1$ means that, once we identify
$$(\ZZ[\P^{-1}]^\times)^n\cong (\ZZ/2)^n \oplus (\ZZ^n)^{|\P|},$$
$\operatorname{im}(\varphi)$~is generated by the~$n-1$ tuples of the form
$(0,\ldots,0,1,1,0,\ldots,0) \in (\ZZ/2)^n$ and the $(n-1)|\P|$ tuples of the form
$$(\underbrace{\underbrace{(0,\ldots,0)}_{\text{$n$~entries}},\ldots,\underbrace{(0,\ldots,0)}_{\text{$n$~entries}},\underbrace{(0,\ldots,0,1,-1,0,\ldots,0)}_{\text{$n$~entries}},\underbrace{(0,\ldots,0)}_{\text{$n$~entries}},\ldots,\underbrace{(0,\ldots,0)}_{\text{$n$~entries}}}_{\text{$|\P|$~tuples}})$$
in $(\ZZ^n)^{|\P|}$. 

The $\ZZ$-free part of~$\operatorname{im}(\varphi)$ is thus isomorphic to $(\ZZ^{n-1})^{|\P|}$, with basis the above tuples. The basis of $\Hom(\operatorname{im}(\varphi), \RR) \cong \Hom((\ZZ^{n-1})^{|\P|}, \RR)$ dual to it corresponds under $\varphi^*$ to~$b$.
\end{proof}

Define $C\subseteq \Hom(\Gamma, \RR)$ as the set of nonzero characters $\sum_{\chi\in b}\lambda_\chi \chi$ with all coefficients $\lambda_\chi$ nonnegative (so $C$~is obtained from a closed cone in $\Hom(\Gamma, \RR)$ by removing the vertex point~$0$). For each $k\ge 0$, we also denote by~$C^{(k)}$ the subset of~$C$ with at most $k$~nonzero coefficients. Schesler's description of the~$\Sigma^k(\Gamma)$ \cite[Theorems 1~and~2]{S23} goes as follows:

\begin{thm}[The homotopical $\Sigma$-sets of~$\Gamma$] \label{thm:Schesler}
Let $k\in\NN$. Then:
\begin{enumerate}
    \item $\Sigma^\infty (\Gamma) = \Hom(\Gamma, \RR) \setminus C$ (in particular, $\Gamma$~is of type $\mathrm F_\infty$).
    \item $\Sigma^k(\Gamma) \subseteq \Hom(\Gamma, \RR) \setminus C^{(k)}$.
    \item If every $p\in \P$ satisfies $p\ge 2^{n-2}$, then $\Sigma^k(\Gamma) = \Hom(\Gamma, \RR) \setminus C^{(k)}$.
\end{enumerate}
\end{thm}

\subsubsection{The $\Sigma$-sets of~$G$}

In order to apply Theorem~\ref{thm:main} to describe the $\TopS^k(\G)$, we need to establish finiteness properties on~$\Lambda$, which we shall do now.

\begin{lem}[The finiteness properties of~$\Lambda$]\label{lem:Lambda_F_infty}
    $\Lambda$ is of type~$\mathrm F_\infty$.
\end{lem}
\begin{proof}
    The diagonal of each matrix in~$\Gamma$ is a sequence of $\pm1$'s that multiply to~$1$. The kernel~$U$ of the map $\Gamma \to (\ZZ/2)^n$ reading this sequence of signs is thus a finite index subgroup of~$\Gamma$, and so we need only show that $U$~is of type~$\mathrm F_\infty$.
    
    Now, $U$~is the middle term of the short exact sequence
    $$1 \to \begin{pmatrix}
        1 &0 &\ZZ&\ZZ&\cdots&\ZZ\\
         & 1  & 0 & \ZZ& \cdots &\ZZ\\
         &&\ddots&\ddots&\ddots&\vdots\\
         &&&1&0&\ZZ\\
        &&&&1&0\\
         &&&&&1
    \end{pmatrix}
    \to 
    \begin{pmatrix}1 &\ZZ &\ZZ&\ZZ&\cdots&\ZZ\\
         & 1  & \ZZ & \ZZ& \cdots &\ZZ\\
         &&\ddots&\ddots&\ddots&\vdots\\
         &&&1&\ZZ&\ZZ\\
        &&&&1&\ZZ\\
         &&&&&1
    \end{pmatrix} \to \ZZ^{n-1} \to 1,$$ where the map onto~$\ZZ^{n-1}$ reads the entries above the diagonal.
    Since $\ZZ^{n-1}$~is of type~$\mathrm F_\infty$ and finiteness properties pass to extensions, one is reduced to showing that the group~$U'$ on the left is of type~$\mathrm{F_\infty}$. To do that we may similarly surject $U'\onto \ZZ^{n-2}$ by reading the entries in the second off-diagonal. Proceeding inductively in this fashion one eventually reaches the trivial group, which is of course of type~$\mathrm F_\infty$.
\end{proof}

Recall the description of $\G$ as~$G/K$ from Proposition~\ref{prop:slcpt_description}.
Using the same argument as in Lemma~\ref{lem:unipotent_vanishing} and Propostion~\ref{prop:Hom_basis} (or appealing to Corollary~\ref{cor:char_space}), we see that $\Homtop(G, \RR)$ has a basis~$\bar b$ comprised of characters $\bar\chi_{k,p} \colon G \to \RR$ with $1 \le k\le n-1$ and $p\in\P$, whose definition is verbatim the same as that of the~$\chi_{k,p}$ from Proposition~\ref{prop:Hom_basis}.

Using the canonical identification
$\Homtop(G, \RR) \cong \Homtop(\G, \RR)$
given by Corollary~\ref{cor:char_space}, we define the ``vertex-less cone'' $\bar C \subseteq \Homtop(\G, \RR)$ as in Section~\ref{sec:Sigma_of_Gamma}, and its $k$-dimensional strata $\bar C^{(k)}$. The fact that $\Lambda$~is of type $\mathrm F_\infty$ (Lemma~\ref{lem:Lambda_F_infty}) and our main result (Theorem~\ref{thm:main}) allow us to transfer Schesler's  Theorem~\ref{thm:Schesler} to~$\G$. We thus conclude:

\begin{prop}[The homotopical $\Sigma$-sets of~$\G$]\label{prop:Schesler_main}
    Let $k\in\NN$. Then:
    \begin{enumerate}
        \item $\TopS^\infty (\G) = \Homtop(\G, \RR) \setminus \bar C$.
        \item $\TopS^k(\G) \subseteq \Homtop(\G, \RR) \setminus \bar C^{(k)}$.
        \item If every $p\in \P$ satisfies $p\ge 2^{n-2}$, then $\TopS^k(\G) = \Homtop(\G, \RR) \setminus \bar C^{(k)}$.
    \end{enumerate}
\end{prop}
   
\appendix 
\section{On the definition of the homological $\Sigma$-sets}
\label{sec:appendix}

The definition of the homological $\Sigma$-sets given in Section~\ref{sec:Sigma} is not verbatim the same as the one introduced  by Bux, Hartmann and the second author \cite[Definition~5.1]{BHQ24b}. 
The purpose of this appendix is to give a faithful account of that definition and connect it to Definition~\ref{dfn:sigma} of the present article (though this equivalence is by no means novel to experts). We will make use of some basic notions in coarse geometry, which can be found, for example, in the book of Leitner and Vigolo \cite[Chapter~2]{LV23}.

For the rest of this appendix, fix a locally compact Hausdorff group~$G$ and a commutative ring~$R$.

The topology of~$G$ induces a coarse structure as follows: for each $C\in \C(G)$, we define
$$\Delta_C:=\{(x,y)\in G\times G \mid x^{-1}y\in C\},$$
and the desired coarse structure on~$G$ is generated by $\{\Delta_C \mid C\in \C(G)\}$. Group multiplication and inversion, being continuous, are controlled maps for this coarse structure, making~$G$ a coarsified set-group \cite[p.~56]{LV23}.

For each $n\in \NN$, the original definition of~$\TopS^n(G;R)$ is given under the additional assumption that $G$~is \textbf{$\sigma$-compact}, that is, $G$~is a countable union of compact subspaces (for example, this is easily seen to hold if $G$~is compactly generated). This permits expressing its coarse structure using a metric:

\begin{lem}[Metric from coarse structure]
    If $G$~is $\sigma$-compact, then its coarse structure is induced by a $G$-bi-invariant metric.
\end{lem}

This statement is essentially the second part of a lemma in an earlier paper \cite[Lemma~4.2]{BHQ24b}; here we provide a somewhat more detailed argument. 

\begin{proof}
    The existence of the desired metric follows once we show that the coarse structure is countably generated \cite[Lemma~8.2.1]{LV23}.

    Using $\sigma$-compactness, write $G=\bigcup_{i\in \NN} C_i$ with $C_i\in \C(G)$ and assume (by adding to each~$C_i$ the union of its predecessors) that this is an ascending union. Since locally compact Hausdorff groups are Baire spaces \cite[Chapter~IX, Theorem~1]{Bou98}, $G$~is not a countable union of subsets with empty interior. Thus, one of the~$C_i$, call it~$K$, has nonempty interior~$U$. Now the sets $V_i:= C_iU$ form an ascending open cover of~$G$, so each compact subset of~$G$ is contained in some~$V_i\subseteq C_iK$. On the other hand, $C_iK$~is itself compact. From here it follows at once that the coarse structure on~$G$ is generated by the (countably many) controlled sets~$\Delta_{C_iK}$.
\end{proof}

It should be stressed that such a metric~$d$ does not generally induce the topology of~$G$. For an example of such~$d$ one should rather consider the word metric with respect to some compact generating set (in the case where $G$~is compactly generated).

Once we choose a metric~$d$ inducing the coarse structure of~$G$, we define, for each $r\in \NN$, the \textbf{Vietoris-Ripps complex} $\VR_r(G)$, which is the simplicial subset of~$\E G$ whose $k$-simples are the tuples $(g_0,\ldots,g_k)$ such that the distance~$d(g_i, g_j)$ between each two entries is at most~$r$. These simplicial sets are assembled into the \textbf{Vietoris-Ripps filtration} $(\VR_r(G))_{r\in \NN}$ of~$\E G$.

We now present the original definition of the homological $\Sigma$-sets \cite[Definition~5.1]{BHQ24b}. As usual, given a character $\chi \colon G\to \RR$ and a simplicial subset~$X\subseteq \E G$, we denote by~$X_\chi$ the subcomplex of~$X$ spanned by the vertices with nonnegative $\chi$-value.

\begin{dfn}\label{def:original_homological}
    Suppose $G$~is equipped with a metric inducing its coarse structure. For each $n\in \NN$, membership of a character $\chi \colon G\to \RR$ in $\TopS^n(G;R)$ is defined by:
    \begin{enumerate}
        \item for $n=1$, we say $\chi \in \TopS^1(G;R)$ if for some $r\in \NN$ we have $\redH_0(\VR_r(G)_\chi; R)=0$;
        \item for $n\ge 2$, we say $\chi\in \TopS^n(G;R)$ if $\chi\in \TopS^1(G;R)$ and the filtration $(\VR_r(G)_\chi)_{r\in\NN}$ of $\E G_\chi$ is essentially $(n-1)$-acyclic over~$R$. 
    \end{enumerate}
\end{dfn}

A priori, it might seem like this definition depends on the choice of metric. We will see that this is not the case as we make our way to proving the main result of this appendix:

\begin{prop}[$\TopS^n(G;R)$ via $(G_\chi \cdot \E C)_{C\in \C(G)}$]\label{prop:appendix_main}
    Suppose $G$~is equipped with a metric inducing its coarse structure and let $n\in \NN$. Then, a character $\chi\colon G\to \RR$ lies in $\TopS^n(G;R)$ (in the sense of Definition~\ref{def:original_homological}) if and only if the filtration $(G_\chi \cdot \E C)_{C\in\C(G)}$ of~$\E G$ is essentially $(n-1)$-acyclic over~$R$.
\end{prop}

Proposition~\ref{prop:appendix_main} justifies our working definition of $\TopS^n(G;R)$ in the body of the paper. Note that the scope of Definition~\ref{dfn:sigma} is extended to locally compact Hausdorff groups whose coarse structure is not metrizable. In those cases, the group $G$~under question is not compactly generated, and Definition~\ref{dfn:sigma} yields $\TopS^n(G;R)=\emptyset$ for all $n\in \NN$.  

Let us first simplify Definition~\ref{def:original_homological} by getting rid of the case distinction:

\begin{lem}[$\TopS^1(G;R)$ via essential acyclicity] \label{lem:simplified_dfn}
    Suppose $G$~is equipped with a metric inducing its coarse structure. Then, for every $n\in \NN$, a character $\chi \colon G \to \RR$ lies in $\chi \in \TopS^n(G;R)$ if and only if $(\VR_r(G)_\chi)_{r\in\NN}$ is essentially $(n-1)$-acyclic over~$R$.
\end{lem}
\begin{proof}
    We need to show that the directed system of $R$-modules $(\tilde \h_0(\VR_r(G)_\chi; R))_{r\in\NN}$ is essentially trivial if and only if some of its elements is~$0$.
    The key observation is that the $0$-skeleton of $\VR_r(G)_\chi$ is independent of~$r$ (explicitly, it always consists of $G_\chi$). Thus, for every $r,s\in\NN$ with $r\le s$, the lower dimensions of the inclusion-induced map of simplicial chain complexes have the form
   $$\begin{tikzcd}
        \Ch_1(\VR_r(G)_\chi;R)\ar[r]\ar[d, hook]& \Ch_0(\VR_r(G)_\chi;R)\ar[r]\ar[d, equals] & R\ar[d, equals]\\
        \Ch_1(\VR_s(G)_\chi;R)\ar[r]&  \Ch_0(\VR_s(G)_\chi;R)\ar[r] & R
    \end{tikzcd}.$$
    
    It is straightforward to see that
    if the upper row is exact, then so is the bottom row, which tells us that if for some~$r$ we have $\tilde \h_0(\VR_r(G)_\chi; R)=0$, then also the homology $R$-modules indexed by $s\ge r$ vanish. In particular, the system is essentially trivial.

    Conversely, if the system is essentially trivial, then in particular some pair of integers $r\le s$ induces a trivial map on homology. It follows that the bottom row of the corresponding diagram is exact, that is, $\tilde \h_0(\VR_s(G)_\chi; R)=0$.
\end{proof}

The following lemma has as a direct consequence that Definition~\ref{def:original_homological} is indeed independent of the choice of metric.

\begin{lem}[Filtrations of $\E G_\chi$] \label{lem:filter_EGchi}
    Suppose $G$~is equipped with a metric inducing its coarse structure. Then the following filtrations of $\E G_\chi$ are cofinal:
    \[(\VR_r(G)_\chi)_{r\in \NN} \quad \text{and} \quad ((G\cdot \E C)_\chi)_{C\in\C(G)}.\]
\end{lem}

\begin{proof}
    We will show the case $\chi=0$, that is, that the filtrations
    \[(\VR_r(G))_{r\in \NN} \quad \text{and} \quad (G\cdot \E C)_{C\in\C(G)}\]
    of~$\E G$ are cofinal. Clearly the stated lemma will then follow as a consequence.

    We will prove that both of these filtrations are cofinal to a third one. For each $C\in\C(G)$, define the simplicial subset $\VR_C(G)\subseteq\E C$, whose $k$-simplices are the tuples of the form $(g_0, \ldots, g_k)$ such that for all $i,j\in \{0, \ldots, k\}$ we have $g_i^{-1}g_j \in C$. The new filtration of~$\E G$ we shall consider is $(\VR_C(G))_{C\in \C(G)}$.

    The fact that $(G\cdot \E C)_{C\in\C(G)}$ and $(\VR_C(G))_{C\in \C(G)}$ are cofinal is the first item of a lemma in the article on homological $\Sigma$-sets \cite[Lemma~4.1]{BHQ24b}.

    As for the comparison between $(\VR_C(G))_{C\in \C(G)}$ and $(\VR_r(G))_{r\in \NN}$, the fact that they are cofinal is precisely a rephrasing of the fact that the metric defining the Vietoris-Ripps filtration is the one induced by the coarse structure of~$G$.
\end{proof}

\begin{proof}[Proof of Proposition~\ref{prop:appendix_main}]
    Lemmas \ref{lem:simplified_dfn}~and~\ref{lem:filter_EGchi} combined tell us that $\chi \in \TopS^n(G;R)$ if and only if the filtration $((G\cdot \E C)_\chi)_{C\in\C(G)}$ of~$\E G_\chi$ is essentially $(n-1)$-acyclic over~$R$. In the article on homotopical $\Sigma$-sets, it is shown that the filtrations $((G\cdot \E C)_\chi)_{C\in\C(G)}$ and $(G_\chi \cdot \E C)_{C\in\C(G)}$ are simplicially homotopy-equivalent \cite[Lemma~3.7]{BHQ24a}. Hence, by Lemma~\ref{lem:homotopyequiv}, they have the same essential acyclicity properties.
\end{proof}

\printbibliography

\end{document}